\documentclass[11pt,oneside]{article}

\usepackage[utf8]{inputenc}
\usepackage[T1]{fontenc}
\usepackage[english]{babel}

\usepackage[tracking=true,expansion=true,protrusion=true]{microtype}
\usepackage{amsthm} 
\usepackage{newpxtext,newpxmath} 
\usepackage[a4paper,margin=1in]{geometry}
\usepackage{fancyhdr}
\usepackage{xcolor}
\definecolor{linkcol}{HTML}{0B4F6C}
\definecolor{citecol}{HTML}{5C7CFA}
\definecolor{urlcol}{HTML}{0A7C66}
\usepackage[pagebackref=true,colorlinks=true,linkcolor=linkcol,citecolor=citecol,urlcolor=urlcol]{hyperref}
\usepackage[nameinlink,capitalize,noabbrev]{cleveref}

\usepackage{accents}
\usepackage{graphicx}
\usepackage{mathbbol}
\usepackage{tikz-cd}

\usepackage{tocloft}

\usepackage{titlesec}
\titleformat{\section}
  {\normalfont\LARGE\bfseries}
  {Lecture~\thesection:}{0.5em}{}
\titlespacing*{\section}{0pt}{1.0em plus .3em}{0.6em}

\titleformat{\subsection}
  {\normalfont\large\bfseries}
  {\S\thesubsection}{0.75em}{}
\titlespacing*{\subsection}{0pt}{0.6em}{0.3em}

\titleformat{\subsubsection}
  {\normalfont\normalsize\bfseries}
  {\S\thesubsubsection}{0.6em}{}
\titlespacing*{\subsubsection}{0pt}{0.5em}{0.2em}

\usepackage{mathtools}
\usepackage{bbm}
\usepackage{thmtools}
\usepackage{float}
\usepackage{tikz}
\usepackage{pgfmath}
\usetikzlibrary{decorations.pathreplacing,calc,arrows.meta,decorations.pathmorphing}

\usepackage[sort&compress]{natbib}
\numberwithin{equation}{section}

\declaretheoremstyle[
  headfont=\bfseries,
  bodyfont=\itshape,
  spaceabove=8pt,spacebelow=6pt,
  headpunct={.},
  postheadspace=0.5em
]{plainclean}
\declaretheoremstyle[
  headfont=\bfseries,
  bodyfont=\normalfont,
  spaceabove=8pt,spacebelow=6pt,
  headpunct={.},
  postheadspace=0.5em
]{defnclean}
\declaretheoremstyle[
  headfont=\bfseries\itshape,
  bodyfont=\normalfont\itshape,
  spaceabove=6pt,spacebelow=6pt,
  headpunct={.},
  postheadspace=0.5em
]{remarkclean}

\theoremstyle{plainclean}
\newtheorem{theorem}{Theorem}[section]

\renewcommand*{\thethmlet}{\Alph{thmlet}}
\newtheorem{proposition}[theorem]{Proposition}

\newtheorem{lemma}[theorem]{Lemma}

\newtheorem{corollary}[theorem]{Corollary}

\newenvironment{theorembox}[1][]
{\refstepcounter{theorem}\begin{tcolorbox}[colback=magenta!2!white,colframe=magenta!30!black,colbacktitle=magenta!5!white,coltitle=purple!90!black,boxrule=0.5pt,arc=3pt,left=8pt,right=8pt,top=6pt,bottom=6pt,fonttitle=\bfseries\boldmath,title=Theorem \thetheorem\ifx\relax#1\relax\else\ (#1)\fi]}
{\end{tcolorbox}}

\newenvironment{theoremconjecturebox}[1][]
{\refstepcounter{theorem}\begin{tcolorbox}[colback=magenta!2!white,colframe=magenta!30!black,colbacktitle=magenta!5!white,coltitle=purple!90!black,boxrule=0.5pt,arc=3pt,left=8pt,right=8pt,top=6pt,bottom=6pt,fonttitle=\bfseries\boldmath,title=Theorem / Conjecture \thetheorem\ifx\relax#1\relax\else\ (#1)\fi]}
{\end{tcolorbox}}

\newenvironment{propositionbox}[1][]
{\refstepcounter{proposition}\begin{tcolorbox}[colback=green!2!white,colframe=green!30!black,colbacktitle=green!5!white,coltitle=green!40!black,boxrule=0.5pt,arc=3pt,left=8pt,right=8pt,top=6pt,bottom=6pt,fonttitle=\bfseries\boldmath,title=Proposition \theproposition\ifx\relax#1\relax\else\ (#1)\fi]}
{\end{tcolorbox}}

\newenvironment{thmletbox}[1][]
{\refstepcounter{thmlet}\begin{tcolorbox}[colback=blue!2!white,colframe=blue!30!black,colbacktitle=blue!5!white,coltitle=blue!90!black,boxrule=0.5pt,arc=3pt,left=8pt,right=8pt,top=6pt,bottom=6pt,fonttitle=\bfseries\boldmath,title=Theorem \thethmlet\ifx\relax#1\relax\else\ (#1)\fi]}
{\end{tcolorbox}}

\newenvironment{propletbox}[1][]
{\refstepcounter{proplet}\begin{tcolorbox}[colback=orange!2!white,colframe=orange!30!black,colbacktitle=orange!5!white,coltitle=orange!90!black,boxrule=0.5pt,arc=3pt,left=8pt,right=8pt,top=6pt,bottom=6pt,fonttitle=\bfseries\boldmath,title=Proposition \theproplet\ifx\relax#1\relax\else\ (#1)\fi]}
{\end{tcolorbox}}

\newenvironment{lemletbox}[1][]
{\refstepcounter{lemlet}\begin{tcolorbox}[colback=cyan!2!white,colframe=cyan!30!black,colbacktitle=cyan!5!white,coltitle=cyan!90!black,boxrule=0.5pt,arc=3pt,left=8pt,right=8pt,top=6pt,bottom=6pt,fonttitle=\bfseries\boldmath,title=Lemma \thelemlet\ifx\relax#1\relax\else\ (#1)\fi]}
{\end{tcolorbox}}

\theoremstyle{defnclean}
\newtheorem{definition}[theorem]{Definition}
\newtheorem{example}[theorem]{Example}
\newtheorem{exercise}[theorem]{Exercise}
\theoremstyle{remarkclean}
\newtheorem{remark}[theorem]{Remark}

\theoremstyle{defnclean}
\newtheorem{notation}[theorem]{Notation}
\newtheorem{terminology}[theorem]{Terminology}
\newtheorem{convention}[theorem]{Convention}
\newtheorem{upshot}[theorem]{Upshot}
\newtheorem{problem}[theorem]{Problem}

\usepackage[many]{tcolorbox}
\tcbset{
  colback=white,
  colframe=black!10!linkcol,
  boxrule=0.6pt,
  arc=3pt,
  left=8pt,right=8pt,top=6pt,bottom=6pt
}

\newenvironment{disclaimer}{%
  \begin{tcolorbox}[colback=red!5!white,colframe=red!75!black,boxrule=0.8pt,arc=3pt,left=8pt,right=8pt,top=6pt,bottom=6pt]
  \textbf{Disclaimer.} }{\end{tcolorbox}}

\usepackage{enumitem}
\setlist{itemsep=0.2em,topsep=0.2em}

\newcommand{\authorname}{\textsc{Kyler Siegel}\footnote{K.S. is partially supported by NSF grant DMS-2506639 and a Simons Foundation Travel Support for Mathematicians Award.}}
\newcommand{\affil}{\textsc{Les Marécottes, Switzerland}}
\newcommand{\lecturedate}{September 2025}
\title{\vspace{-1.5em}\textbf{Lectures on Stabilized Ellipsoid Embeddings}}
\author{\authorname\\ \affil\\ \lecturedate}
\date{}

\newcommand{\R}{\mathbb{R}}
\newcommand{\ra}{\rightarrow}
\newcommand{\om}{\omega}
\newcommand{\1}{\mathbbm{1}}
\newcommand{\fl}{\op{fl}}
\newcommand{\op}[1]{{\operatorname{#1}}}
\newcommand{\veca}{{\vec{a}}}
\newcommand{\T}{\mathbb{T}}
\newcommand{\C}{\mathbb{C}}
\newcommand{\vecb}{{\vec{b}}}
\newcommand{\hooksymp}{\overset{s}\hookrightarrow}
\definecolor{darkmagenta}{rgb}{0.55, 0.0, 0.55}
\newcommand{\hl}[1] {{\boldmath\textbf{{\color{darkmagenta}#1}}}}
\newcommand{\emb}{\mathcal{E}}
\newcommand{\fib}{{\op{Fib}}}
\newcommand{\Q}{\mathbb{Q}}
\newcommand{\weight}{\op{wt}}
\newcommand{\bl}{{\op{Bl}}}
\newcommand{\CP}{\mathbb{CP}}
\newcommand{\ovl}{\overline}
\newcommand{\calC}{\mathcal{C}}
\newcommand{\pd}{\op{PD}}
\newcommand{\Z}{\mathbb{Z}}
\newcommand{\exc}{\op{Exc}}
\newcommand{\al}{\alpha}
\newcommand{\NI}{{\noindent}}

\newcommand{\uvl}{\underline}
\newcommand{\de}{\delta}
\newcommand{\area}{\op{area}}
\newcommand{\ind}{\op{ind}}
\newcommand{\vol}{\op{vol}}
\newcommand{\sss}{\vspace{2.5 mm}}
\newcommand{\step}[1]{\par\medskip\noindent\textbf{Step #1}}
\newcommand{\pp}{{\mathfrak{p}}}
\newcommand{\wt}{\widetilde}
\newcommand{\eps}{\varepsilon}
\makeatletter
\newcommand{\doublewtikz}[2][-0.28ex]{%
  \mathchoice
    {\doublewtikz@{\displaystyle}{#1}{#2}}%
    {\doublewtikz@{\textstyle}{#1}{#2}}%
    {\doublewtikz@{\scriptstyle}{#1}{#2}}%
    {\doublewtikz@{\scriptscriptstyle}{#1}{#2}}%
}
\newcommand{\doublewtikz@}[3]{%
  \tikz[baseline=(base.base)]{
    \node[inner sep=0pt, outer sep=0pt] (base) {$#1 #3$};
    \node[inner sep=0pt, outer sep=0pt, yshift=-0.45ex] (t1) at (base.north)
      {$#1 \widetilde{\phantom{#3}}$};
    \node[inner sep=0pt, outer sep=0pt, yshift={#2-0.45ex}] at (t1.north)
      {$#1 \widetilde{\phantom{#3}}$};
  }%
}
\makeatother
\newcommand{\wtt}[1]{\doublewtikz[-0.2ex]{#1}}
\newcommand{\sm}{\op{sm}}
\newcommand{\Om}{{\Omega}}
\newcommand{\vecD}{\vec{D}}
\newcommand{\calJ}{\mathcal{J}}
\newcommand{\vecm}{{\vec{m}}}
\renewcommand{\lll}{\Langle}
\newcommand{\rrr}{\Rangle}
\newcommand{\calM}{\mathcal{M}}
\newcommand{\ga}{\gamma}
\newcommand{\calA}{\mathcal{A}}
\newcommand{\bdy}{\partial}
\newcommand{\pt}{\op{pt}}
\newcommand{\tor}{{\op{tor}}}
\newcommand{\mm}{\mathfrak{m}}
\newcommand{\calT}{\mathcal{T}}
\newcommand{\mut}{\op{mut}}
\newcommand{\gl}{\op{GL}}
\newcommand{\calB}{\mathcal{B}}
\newcommand{\calS}{\mathcal{S}}
\newcommand{\aut}{\op{Aut}}
\newcommand{\frakd}{\mathfrak{d}}
\newcommand{\calD}{{\mathcal{D}}}
\newcommand{\nn}{\mathfrak{n}}
\newcommand{\mon}{\op{Mon}}
\newcommand{\ks}{{\operatorname{{\mathsf{S}}}}}
\newcommand{\ka}{\kappa}
\newcommand{\MM}{\mathbb{M}}
\newcommand{\lan}{\langle}
\newcommand{\ran}{\rangle}
\newcommand{\acc}{\op{acc}}

\begin{document}
\maketitle
\vspace{-1.5em}
\begin{abstract}
These notes are based on a five-part minicourse on stabilized symplectic embeddings given in Les Marécottes, Switzerland during a September 2025 workshop.\footnote{Special thanks to Felix Schlenk and Grisha Mikhalkin for organizing this workshop.}
Our main goal is to explain the recent resolution of the (restricted) stabilized ellipsoid embedding problem by D. McDuff and the author. Along the way we also introduce various other ideas which shed light on the context and hint at possible generalizations.
Some of the concepts covered include sesquicuspidal curves, symplectic inflation, multidirectional tangency constraints, well-placed curves, cluster transformations, Looijenga pairs, toric models, scattering diagrams, and the tropical vertex theorem.

\end{abstract}

\tableofcontents
\bigskip

\begin{disclaimer}
Many of the mathematical statements in these lectures are somewhat informal simplifications; we recommend consulting the original references for more precise formulations whenever possible.
Also, we apologize in advance for any missing citations, and we welcome any comments, questions, or corrections.
\end{disclaimer}


\section{Overview}\label{sec:overview}

\textit{In this lecture, we first briefly introduce Hamiltonian flows, nonsqueezing phenomena, and the ellipsoid embedding problem in \S\ref{subsec:hamiltonian_flows}. We then recall the celebrated McDuff--Schlenk Fibonacci staircase in \S\ref{subsec:fibonacci_staircase}, and give a brief modern overview of ways to derive it. 
In \S\ref{subsec:full_fillings} we recall some connections between four-dimensional ellipsoid embeddings, symplectic ball packings, and symplectic cones, and as an application we establish full symplectic fillings of the four-ball by sufficiently long ellipsoids.
In \S\ref{subsec:stabilized_embeddings}, we introduce stabilized ellipsoid embeddings and formulate our main result, namely Theorem~\ref{thmlet:A}.
Lastly, in \S\ref{subsec:sesquicuspidal_curves} we introduce sesquicuspidal curves as a key tool for obstructing stabilized ellipsoid embeddings.
We end with a brief sketch in \S\ref{subsec:outline} of the remaining lectures.}

\subsection{Hamiltonian flows and ellipsoid embeddings}\label{subsec:hamiltonian_flows}

Consider a Hamiltonian on $\R^{2n}$, i.e. a smooth function $H: \R^{2n} \ra \R$. 
Its {\em symplectic gradient} (a.k.a. {\em Hamiltonian vector field}) is defined by $\nabla_\om H:= J\nabla H$, where $\nabla$ is the usual gradient and $J$ is the block matrix 
$J = 
\begin{pmatrix}
0 & -\1 \\
\1 & 0
\end{pmatrix}.$
Integrating $\nabla_\om H$ gives rise to the flow $\fl_H^t: \R^{2n} \ra \R^{2n}$ for $t \in \R$.
There are various interesting questions one can ask about Hamiltonian flows, for instance:
\begin{itemize}
    \item  periodic orbits (existence, lower bounds, density, \dots)
    \item  integrability versus chaos (exact solvability, KAM tori)
    \item Poincaré recurrence (or ``super-recurrence'')
\end{itemize}
and so on.
These lectures are related to the so-called ``symplectic nonsqueezing phenomena'' and its generalizations. We first recall some basics.

\begin{propositionbox}[``Liouville's theorem'']
The symplectic gradient $\nabla_\om H$ is divergence-free, i.e. $\fl_H^t$ preserves volume.
\end{propositionbox}

\begin{notation}
For $\veca = (a_1,\dots,a_n) \in \R_{>0}^n$, put $E(\veca) := \left\{ \pi\sum\limits_{i=1}^n \frac{1}{a_i}(x_i^2+y_i^2) \leq 1\right\}$.
\end{notation}

\begin{remark}
Note that $E(a,\dots,a)$ is the ball $B^{2n}(a)$ of radius $\sqrt{a/\pi}$.
Also, if $\mu: \C^n \ra \R_{\geq 0}^n$ denotes the moment map for the standard $\T^n$-action, i.e. $\mu(z_1,\dots,z_n) = (\pi |z_1|^2,\dots,\pi |z_n|^2)$, then $\mu(E(a_1,a_2))$ is the right triangle with vertices $(0,0),(a_1,0),(0,a_2)$ as in Figure~\ref{fig:right_tri}.
\end{remark}

\begin{figure}[h]\label{fig:right_tri}
\centering
\begin{tikzpicture}[scale=1]
\draw[->] (-0.5,0) -- (4,0) node[right] {$x_1$};
\draw[->] (0,-0.5) -- (0,3) node[above] {$x_2$};

\draw[thick, fill=blue!20] (0,0) -- (3,0) -- (0,2) -- cycle;

\fill[black] (0,0) circle (2pt) node[below left] {$(0,0)$};
\fill[black] (3,0) circle (2pt) node[below] {$(a_1,0)$};
\fill[black] (0,2) circle (2pt) node[left] {$(0,a_2)$};

\node at (1.5,-0.3) {$a_1$};
\node at (-0.3,1) {$a_2$};

\draw (0.3,0) -- (0.3,0.3) -- (0,0.3);
\end{tikzpicture}
\caption{The image of $E(a_1,a_2)$ under the moment map $\mu: \C^2 \ra \R_{\geq 0}^2$.}
\end{figure}
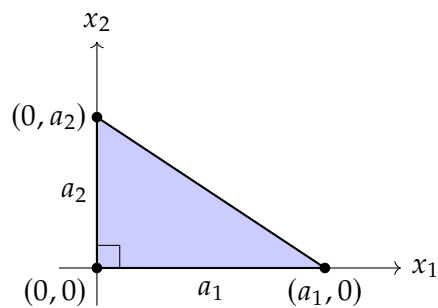

\begin{notation}
  Put $E(\veca) \hooksymp E(\vecb)$ if $\fl_H^1(E(\veca)) \subset E(\vecb)$ for some (time-dependent) $H: \R^{2n} \times [0,1] \ra \R$.
\end{notation}
\begin{remark}
This turns out to be equivalent to abstract symplectic embedding, i.e. a smooth embedding preserving the standard symplectic form, thanks to the ``extension after restriction principle'' -- see e.g. \cite[\S4.4]{Schlenk_old_and_new}. The discrepancy between time-dependent and time-independent Hamiltonian flows is more subtle -- see e.g. \cite{polterovich2016autonomous}.
\end{remark}

\begin{theorembox}[Gromov `85 \cite{gromov1985pseudo}]
If $E(\veca) \hooksymp E(\vecb)$, then we must have $\min \{a_i\} \leq \min \{b_i\}$.
\end{theorembox}

\begin{definition}
For a symplectic manifold $M^{2n}$, the $\hl{ellipsoid embedding function}$ is 
\begin{align*}
\emb_M: \R_{>0}^n \ra \R_{>0},\;\;\;\;\; \emb_M(\veca) := \inf \{ c \; | \; E(\tfrac{1}{c} \cdot \veca) \hooksymp M \}.  
\end{align*}
\end{definition}
\begin{remark}
 Note that we have $\emb_M(c \cdot \veca) = c \emb_M(\veca)$ for any scalar $c \in \R_{>0}$, i.e. $\emb_M(\veca)$ is really a function of just $n-1$ variables.
\end{remark}

\begin{problem}[``ellipsoid embedding problem'' (EEP)]
Compute $\emb_M(\veca)$ for $M = E(\vecb)$.
\end{problem}
\begin{problem}[``restricted ellipsoid embedding problem'' (REEP)]
Compute $\emb_M(\veca)$ for $M = B^{2n}$.
\end{problem}

Here are a few initial comments.
\begin{itemize}
  \item The ellipsoid embedding problem is trivial for $n=1$.
  \item The ellipsoid embedding problem is ``solved'' for $n=2$ by the work of many people (McDuff, Hutchings, Schlenk, Hofer, Biran, Polterovich, Li, \dots). Here ``solved'' means reduced to a deceptively simple combinatorial criterion.
  \item The solution to the restricted ellipsoid embedding problem for $n=2$ was worked out in complete detail by McDuff--Schlenk \cite{McDuff-Schlenk_embedding}.
  \item To our knowledge there is not even a conjecture for the restricted ellipsoid embedding problem for $n \geq 3$ (i.e. a formula for the function $\emb_{B^6}(a_1,a_2,a_3)$).
\end{itemize}

\subsection{The symplectic Fibonacci staircase}\label{subsec:fibonacci_staircase}

\begin{theorembox}[McDuff--Schlenk \cite{McDuff-Schlenk_embedding}]\label{thm:McSch}

The ellipsoid embedding function $\emb_{B^4}(1,a)$ of the four-dimensional ball $B^4$ is as follows:

\begin{figure}[H]
\centering
\includegraphics[width=1\textwidth]{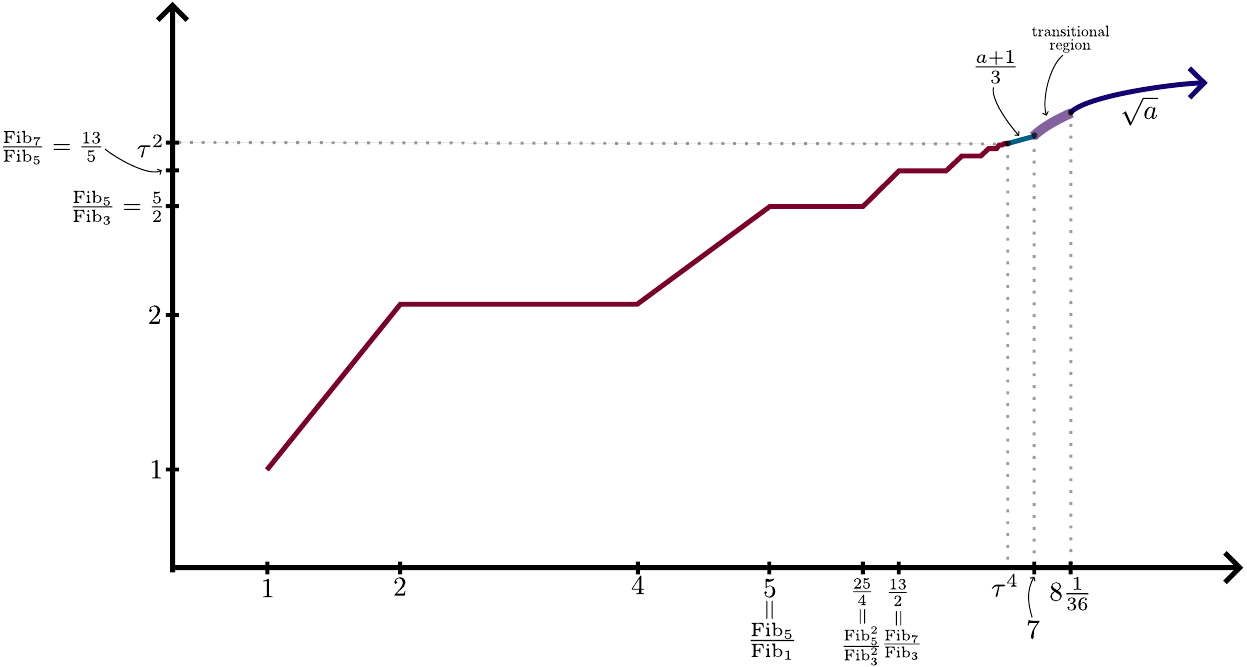}
\caption{The ellipsoid embedding function $\emb_{B^4}(1,a)$ (not to scale).}
\label{fig:fib_stairs}
\end{figure}
\end{theorembox}

While Figure~\ref{fig:fib_stairs} is just a cartoon, the function $\emb_{B^4}(1,a)$ is explicitly computed in \cite{McDuff-Schlenk_embedding} and it includes the following noteworthy features:
\begin{itemize}
  \item it is an ``infinite Fibonacci staircase'' for $a \in [1,\tau^4]$, with infinitely many steps accumulating at $\tau^4 := \tfrac{1}{2}(7+3\sqrt{5}) \approx 6.85$, where the $x$ and $y$ coordinates of the outer and inner corners are ratios of Fibonacci numbers (or their squares)
  \item for $a \in [8\tfrac{1}{36},\infty)$, we have $\emb_{B^4}(1,a) = \sqrt{a}$, i.e. volume is the only embedding obstruction.
  \item it is the linear function $\tfrac{1}{3}(a+1)$ on the segment $a \in [\tau^4,7]$, and in the range $a \in [7,8\tfrac{1}{36}]$ it exhibits transitional behavior.
\end{itemize}

Let $\fib_i$ denote the $i$th Fibonacci number:
\[\begin{array}{|c|c|c|c|c|c|c|c|c|c|c|c|}
\hline
i & -1 & 0 & 1 & 2 & 3 & 4 & 5 & 6 & 7 & \cdots \\ \hline
\fib_i & 1 & 0 & 1 & 1 & 2 & 3 & 5 & 8 & 13 & \cdots \\ \hline
\end{array}
\]
\begin{exercise}\label{ex:determined_by_corners}
Show that $\emb_{B^4}(1,a)$ is nondecreasing and $\frac{1}{a}\cdot\emb_{B^4}(1,a)$ is nonincreasing. Conclude that $\emb_{B^4}(1,a)$ for $a \in [1,\tau^4]$ is entirely characterized by:
\begin{itemize}
  \item embedding obstructions: 
  \begin{equation}
    \emb_{B^4}\left(1,\frac{\fib_{2k+5}}{\fib_{2k+1}}\right) \geq \frac{\fib_{2k+5}}{\fib_{2k+3}} \label{eq:emb_obstr}
  \end{equation}
  \item embedding constructions: 
  \begin{equation}
    \emb_{B^4}\left(1,\frac{\fib_{2k+3}^2}{\fib_{2k+1}^2}\right) \leq \frac{\fib_{2k+5}}{\fib_{2k+1}} \label{eq:emb_constr}
  \end{equation}
\end{itemize}
\end{exercise}

There are multiple ways of establishing \eqref{eq:emb_obstr} and \eqref{eq:emb_constr}.
For \eqref{eq:emb_obstr}, all of the main approaches are based on \hl{pseudoholomorphic curves}.\footnote{It is, however, interesting to ask how much of this story could be recovered using generating functions or microlocal sheaves.}
While the method of extracting symplectic embedding obstructions from pseudoholomorphic curves is by now fairly streamlined, there are many different approaches for {\em producing} such curves, for instance:
\begin{itemize}
  \item applying Cremona transformations to basic exceptional curves
  \item embedded contact homology
  \item cluster symmetries.
\end{itemize}
In these lectures we will primarily explain the last bullet above, as it is the most directly connected to our primary goal.

For \eqref{eq:emb_constr}, there are several methods for constructing symplectic embeddings, most notably:
\begin{itemize}
  \item inflation along pseudoholomorphic curves
  \item almost toric fibrations.
\end{itemize}
While almost toric fibrations (ATFs) provide a very visually and conceptually appealing approach to the embeddings for \eqref{eq:emb_constr}, we will primarily focus on the first bullet above in these lectures, partly because the ATF approach to embeddings is already discussed very thoroughly elsewhere (see e.g. \cite{symington71four,evans2023lectures,cristofaro2020infinite,casals2022full}) and partly because it allows us to reduce both \eqref{eq:emb_obstr} and \eqref{eq:emb_constr} to the existence of certain singular symplectic curves in the complex projective plane.
There are then multiple methods for constructing the curves relevant for inflation approach to \eqref{eq:emb_constr}: Seiberg--Witten theory, tropical methods (see e.g. \cite[\S5.3]{mcduff2024singular}), etc.

\subsection{Full fillings by long ellipsoids}\label{subsec:full_fillings}

We associate to each positive rational number $a \in \Q_{>1}$ its \hl{weight sequence} $\weight(a) = (w_1,\dots,w_k)$, whose definition is more or less explained by the ``box diagram'' in Figure~\ref{fig:box_diagram}, which divides a rectangle of height $1$ and width $a$ into a collection of squares.

\begin{example}
For $a = 17/5$, the corresponding weight sequence is $\weight(17/5) = (\underbrace{1,1,1}_3,\underbrace{2/5,2/5}_2,\underbrace{1/5,1/5}_2).$ Note that:
\begin{itemize}
  \item $\frac{17}{5} = 3 + \frac{1}{2 + \frac{1}{2}}$, i.e. the multiplicities of the entries in the weight sequence give the continued fraction expansion
  \item $1 \geq w_1 \geq w_2 \geq \cdots \geq w_k$
  \item $w_1^2 + \cdots + w_k^2 = a$.
\end{itemize}
Each of these properties holds for general $a \in \Q_{>1}$.
\end{example}

\begin{figure}[H]
\centering
\begin{tikzpicture}[scale=4,>=Stealth]

  \pgfmathsetmacro{\H}{1}          
  \pgfmathsetmacro{\W}{17/5}       
  \pgfmathsetmacro{\twf}{2/5}      
  \pgfmathsetmacro{\owf}{1/5}      

  \draw[thick] (0,0) rectangle ({\W},{\H});

  \foreach \x in {1,2,3} {
    \draw[thick] (\x,0) -- (\x,{\H});
  }
  
  \fill[blue!10] (0,0) rectangle (1,{\H});
  \fill[blue!10] (1,0) rectangle (2,{\H});
  \fill[blue!10] (2,0) rectangle (3,{\H});
  
  \draw[thick] (0,0) rectangle (1,{\H});
  \draw[thick] (1,0) rectangle (2,{\H});
  \draw[thick] (2,0) rectangle (3,{\H});

  \draw[thick] (3,{\twf}) -- ({\W},{\twf});       
  \draw[thick] (3,{2*\twf}) -- ({\W},{2*\twf});   

  \draw[thick] ({3+\owf},{2*\twf}) -- ({3+\owf},{\H});
  
  \fill[green!10] (3,0) rectangle ({\W},{\twf});
  \fill[green!10] (3,{\twf}) rectangle ({\W},{2*\twf});
  
  \fill[red!10] (3,{2*\twf}) rectangle ({3+\owf},{\H});
  \fill[red!10] ({3+\owf},{2*\twf}) rectangle ({\W},{\H});
  
  \draw[thick] (3,0) rectangle ({\W},{\twf});
  \draw[thick] (3,{\twf}) rectangle ({\W},{2*\twf});
  
  \draw[thick] (3,{2*\twf}) rectangle ({3+\owf},{\H});
  \draw[thick] ({3+\owf},{2*\twf}) rectangle ({\W},{\H});


  \draw[<->] (0,-0.16) -- ({\W},-0.16)
    node[midway,fill=white,inner sep=1pt] {$\frac{17}{5}$};

  \draw[<->] (-0.08,0) -- (-0.08,1)
    node[midway,fill=white,inner sep=1pt,rotate=90] {$1$};

  \draw[<->] (3,-0.08) -- ({3+\twf},-0.08)
    node[midway,fill=white,inner sep=1pt] {$\frac{2}{5}$};

  \draw[<->] ({\W+0.10},{2*\twf}) -- ({\W+0.10},{\H})
    node[midway,fill=white,inner sep=1pt,rotate=90] {$\frac{1}{5}$};

\end{tikzpicture}
\caption{The box diagram giving the weight sequence for $a=17/5$.}
\label{fig:box_diagram}
\end{figure}
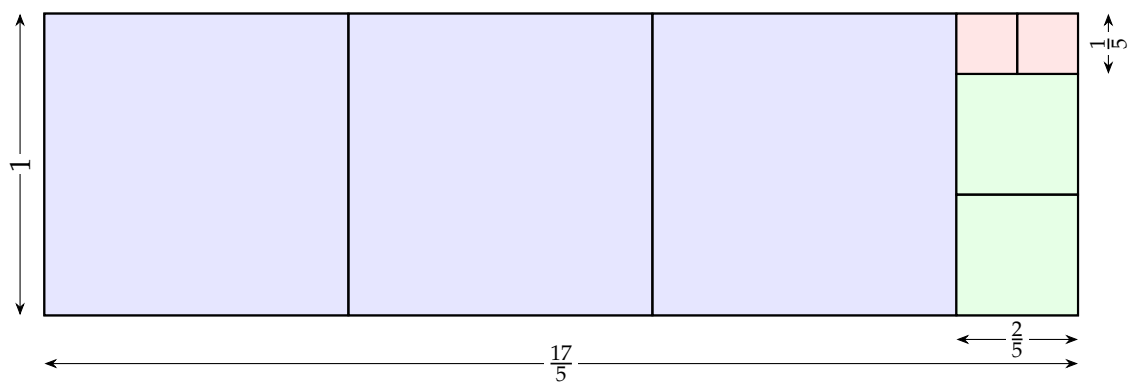

\begin{theorembox}[McDuff \cite{mcduff2009symplectic}]
For each $a \in \Q_{> 1}$, say with weight sequence $\weight(a) = (w_1,\dots,w_k)$, we have 
\begin{align*}
E(1,a) \hooksymp B^4(c) \;\;\;\Longleftrightarrow\;\;\; \bigsqcup_{i=1}^k B^4(w_i) \hooksymp B^4(c).
\end{align*}

\end{theorembox}

\begin{notation}
Put:
\begin{itemize}
  \item $\bl^k\CP^2 := \CP^2 \# (\underbrace{\ovl{\CP}^2 \# \cdots \# \ovl{\CP}^2}_k)$, i.e. the $k$-fold blowup of $\CP^2$ as a smooth oriented manifold
  \item $H_2(\bl^k\CP^2) = \Z\langle \ell,e_1,\dots,e_k\rangle$, with $\ell$ the line class and $e_1,\dots,e_k$ the obvious exceptional classes
  \item $\calC(\bl^k\CP^2) := \left\{ [\omega] \in H^2(\bl^k\CP^2)\; \left|\, 
    \begin{array}{l}
      \omega \text{ symplectic form on } \bl^k\CP^2, \\[0.3em]
      c_1(\omega) = \pd(3\ell - e_1 - \cdots - e_k)
    \end{array}
    \right.\right\}$,\; the ``symplectic cone with standard first Chern class''.
\end{itemize}
\begin{theorembox}[McDuff--Polterovich \cite{mcduff1994symplectic}]
 We have 
 \begin{align*}
\bigsqcup_{i=1}^k B^4(w_i) \hooksymp B^4(c) \;\;\;\Longleftrightarrow\;\;\; \pd(c\ell - w_1e_1 - \cdots - w_ke_k) \in \calC(\bl^k\CP^2).
 \end{align*}
\end{theorembox}

\begin{notation}
Let $\exc(\bl^k\CP^2)$ denote the set of all homology classes $e \in H_2(\bl^k\CP^2)$ such that $e$ is represented by a symplectically embedded sphere with self-intersection number $-1$.
\end{notation}
\begin{remark}
Note that a priori the definition of $\exc(\bl^k\CP^2)$ depends on the symplectic form on $\bl^k\CP^2$, but in fact one can show that
it is independent of the choice of symplectic form in $\calC(\bl^k\CP^2)$ (this uses a Gromov compactness argument and the fact that any two $\om,\om' \in \calC(\bl^k\CP^2)$ are deformation equivalent\footnote{Here we say that two symplectic forms $\om_1,\om_2$ on a smooth manifold $M$ are \hl{deformation equivalent} if there is a homotopy through symplectic forms connecting $\om_1$ to $\phi^*\om_2$, where $\phi$ is a diffeomorphism of $M$. For $M = \bl^k\CP^2$, it is actually not known whether this is equivalent to the existence of a homotopy through symplectic forms connecting $\om_1$ to $\om_2$ (we thank D. McDuff for pointing this out).}).
\end{remark}
\begin{remark}
 Note that for $e = d\ell - m_1e_1 - \cdots - m_ke_k \in \exc(\bl^k\CP^2)$, we have 
\begin{align}
 e \cdot e = d^2 - m_1^2 - \cdots - m_k^2 = -1 \;\;\;\;\;\text{and}\;\;\;\;\;
c_1(e) = 3d - m_1 - \cdots - m_k = 1.
\end{align}
\end{remark}
\begin{theorembox}[\cite{mcduff_from_def_to_iso,biran1997symplectic,li2001uniqueness,li_li_2002}]
We have 
\begin{align*}
\calC(\bl^k\CP^2) = \left\{\al \in H^2(\bl^k\CP^2)\;\left|\; \begin{aligned}
&\hspace{3.5em}\al \cdot \al > 0, \\
&\langle \al,e\rangle > 0 \;\forall\; e \in \exc(\bl^k\CP^2)
\end{aligned}\right.\right\}.
\end{align*}
\end{theorembox}
\begin{corollary}
For $a \in \Q_{> 1}$, say with $\weight(a) = (w_1,\dots,w_k)$, we have
\begin{align*}
\emb_{B^4}(1,a) = \max \left( \sqrt{a}, \sup\left\{ \frac{m_1w_1+\cdots+m_kw_k}{d}\;\Big|\; d\ell - m_1e_1 - \cdots - m_ke_k \in \exc(\bl^k\CP^2)  \right\} \right )
\end{align*}
\end{corollary}
\end{notation}

\begin{lemma}
We have $\emb_{B^4}(1,a) = \sqrt{a}$ for all $a \geq 9$.  
\end{lemma}
\begin{proof}
 It suffices to establish $\frac{m_1w_1 + \cdots + m_kw_k}{d} \leq \sqrt{a}$ for all $d\ell - m_1e_1 - \cdots - m_ke_k \in \exc(\bl^k\CP^2)$.
 Observe that we have
 \begin{align*}
\frac{m_1w_1 + \cdots + m_kw_k}{d} \leq \frac{m_1 + \cdots + m_k}{d} = \frac{3d-1}{d} < 3 < \sqrt{a}.
 \end{align*}
\end{proof}

\subsection{Stabilized ellipsoid embeddings}\label{subsec:stabilized_embeddings}

\begin{definition}
  For a symplectic manifold $M^{2n}$ and $N \in \Z_{\geq 1}$, the \hl{stabilized ellipsoid embedding function} is
  \begin{align}\label{eq:stab_less_unstab}
\emb_M^N: \R_{>0}^n \ra \R_{>0},\;\;\;\;\; \emb_M^N(\veca) := \inf \{ c \; | \; E(\tfrac{1}{c}\cdot \veca)  \times \R^{2N} \hooksymp M \times \R^{2N}\}.    
  \end{align}
\end{definition}
\begin{remark}
  Note that by taking the identity map in the extra dimensions, we always have 
  \begin{align*}
  \emb_M^N(\veca) \leq \emb_M(\veca).
  \end{align*}
\end{remark}
\begin{problem}[``stabilized ellipsoid embedding problem'' (SEEP)]
Compute $\emb_M^N(\veca)$ for $M = E(\vecb)$ and $N \in \Z_{\geq 1}$.
\end{problem}
\begin{problem}[``restricted stabilized ellipsoid embedding problem'' (RSEEP)]
Compute $\emb_M^N(\veca)$ for $M = B^{2n}$ and $N \in \Z_{\geq 1}$.
\end{problem}

\begin{remark}
Note that we can think of $E(a_1,\dots,a_n) \times \R^{2N}$ as $E(a_1,\dots,a_n,\underbrace{\infty,\dots,\infty}_N)$,
and hence the stabilized ellipsoid embedding problem is essentially a special case of the corresponding unstabilized ellipsoid embedding problem (but now in higher dimensions).
\end{remark}

Our primary goal in these lectures is to explain the proof of the following theorem.
\begin{thmletbox}[\cite{sesqui}]\label{thmlet:A}
 For all $N \in \Z_{\geq 1}$, the stabilized ellipsoid embedding of the four-dimensional ball is given by
 \begin{align*}
  \emb_{B^4}^N(1,a) =
  \begin{cases}
    \emb_{B^4}(1,a) &  a \in [1,\tau^4]\\
    \frac{3a}{a+1} & a > \tau^4
  \end{cases}
  \end{align*}
\end{thmletbox}
\NI See Figure~\ref{fig:stab_fib} for an illustration.

\begin{figure}[H]
\centering
\includegraphics[width=0.9\textwidth]{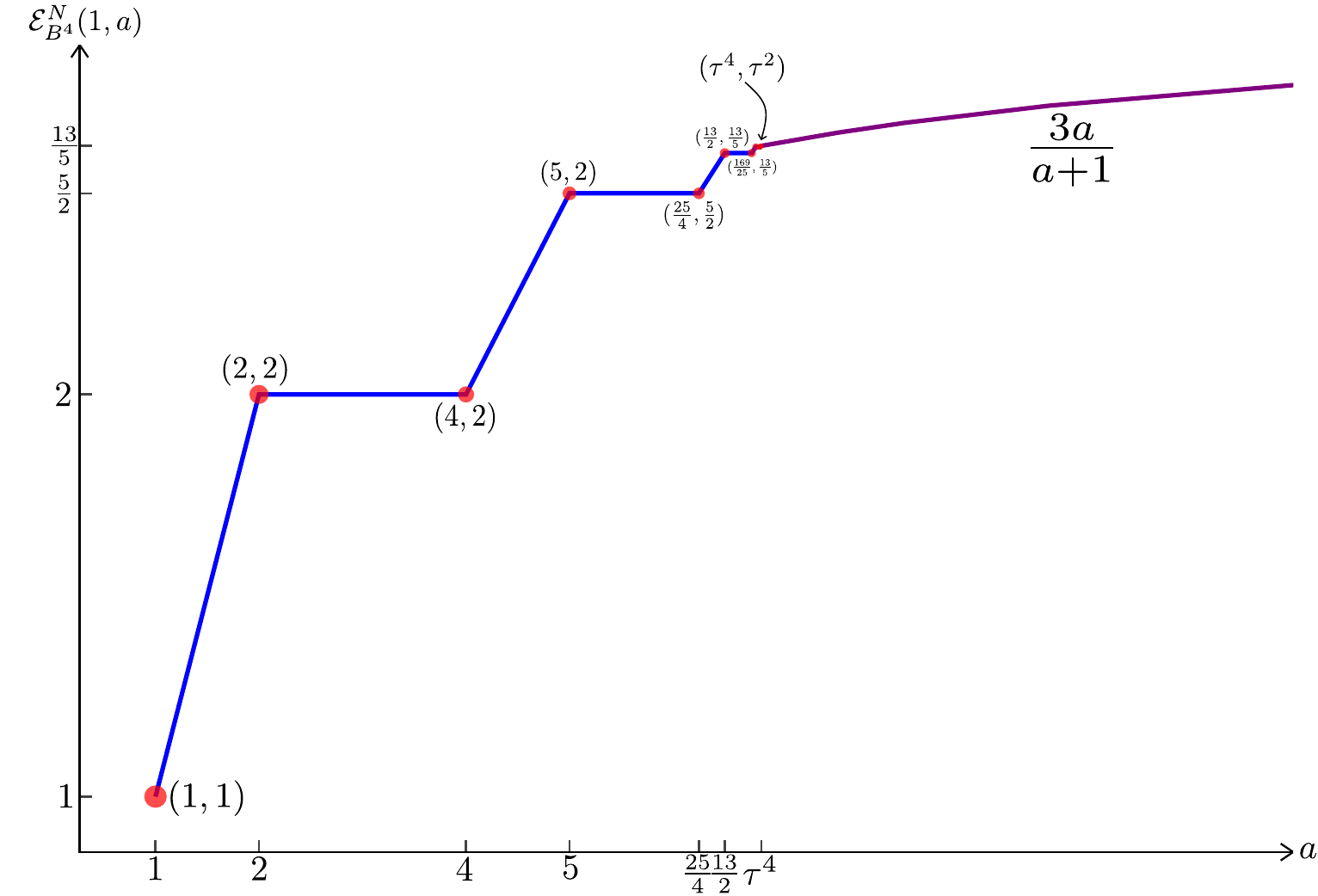}
\caption{The stabilized ellipsoid embedding function $\emb_{B^4}^N(1,a)$ (taken from \cite{sesqui}).}
\label{fig:stab_fib}
\end{figure}

In brief, the function $\emb_{B^4}^N(1,a)$ has a ``subcritical regime'' which is an infinite staircase and a ``supercritical regime'' which is a smooth rational function, with the phase transition at $a = \tau^4$. 
 One can check that the rational function $\tfrac{3a}{a+1}$ meets the infinite Fibonacci staircase precisely at its outer corners and is otherwise above it for $a \in [1,\tau^4)$, whereas it lies below $\sqrt{a}$ for $a > \tau^4$.
Note that the volume bound $\sqrt{a}$ evidently no longer holds for the stabilized problem, i.e. the supercritical behavior is a truly higher dimensional phenomenon.
\begin{remark}
There were important earlier partial results by subsets of Hind, Kerman, Cristofaro-Gardiner, McDuff, and the author (see e.g. \cite{HK,CGH,Ghost,Mint,chscI}).
In particular, the function for $a \in [1,\tau^4]$ was worked out in \cite{CGH}, which by Exercise~\ref{ex:determined_by_corners} and \eqref{eq:stab_less_unstab} boils down to showing that the inequality \eqref{eq:emb_obstr} persists after stabilization.
\end{remark}
The upper bounds for Theorem~\ref{thmlet:A} are given by an explicit ``symplectic folding'' construction due to Hind, which adapts Guth's polydisk embedding construction in \cite{Guth_polydisks}.
\begin{theorembox}[Hind \cite{hind2015some}]
For each $a \in \R_{\geq 1}$ and $N \in \Z_{\geq 1}$ we have 
\begin{align}
E(1,a) \times \R^{2N} \hooksymp B^4(c) \times \R^{2N}\;\;\;\;\;\text{for}\;\;\;\;\; c > \frac{3a}{a+1}.
\end{align}
\end{theorembox}

Thus, our main goal is to find optimal lower bounds for $\emb_{B^4}^N(1,a)$, especially for $a > \tau^4$.

\subsection{Unicuspidal and sesquicuspidal curves}\label{subsec:sesquicuspidal_curves}

Let $p,q$ be positive integers satisfying $\gcd(p,q) = 1$ and $p > q \geq 2$.
\begin{definition}\label{def:sesqui}
A \hl{$(p,q)$-sesquicuspidal symplectic curve} in a symplectic four-manifold $M^4$ is a subset $C \subset M$ such that
\begin{itemize}
  \item $C$ has one $(p,q)$ cusp singularity (i.e. modeled locally on $\{x^p + y^q = 0\} \subset \C^2$)
  \item $C$ has finitely many additional positive ordinary double points (i.e. modeled locally on $\{xy = 0\} \subset \C^2$)
  \item $C$ is otherwise symplectically embedded.
\end{itemize}
\end{definition}
\begin{remark}
  The second condition in Definition~\ref{def:sesqui} holds after a generic perturbation, and the number of double points $\uvl{\de}$ away from the distinguished cusp is fixed by adjunction, namely 
\begin{align}
  \uvl{\de} = \tfrac{1}{2}(C \cdot C + \chi(C) - c_1(C) - (p-1)(q-1)).
\end{align}
In the case $\uvl{\de} = 0$, the curve is called \hl{unicuspidal}.\footnote{The term sesquicuspidal comes from Latin sesqui meaning ``one and a half'', reflecting the fact that sesquicuspidal curves are only mildly more singular than unicuspidal curves.}
Note that we can similarly define $(p,q)$-sesquicuspidal {\em algebraic} curves, which in particular give examples of $(p,q)$-sesquicuspidal symplectic curves.
\end{remark}
\begin{definition}
  We will define the (real) \hl{index} of a $(p,q)$-sesquicuspidal curve $C \subset M$ to be $\ind(C) := 2(c_1(C) - p -q)$.
\end{definition}
We will see later that this is the expected (real) dimension of the moduli space of $(p,q)$-sesquicuspidal curves in a fixed homology class and with fixed ``position'' of the distinguished cusp.
\begin{propletbox}[\cite{SDEP,cusps_and_ellipsoids}]\label{proplet:B}
Given a rational index zero $(p,q)$-sesquicuspidal symplectic curve $C$ in a closed symplectic four-manifold $M^4$ with coprime $p \geq q \geq 1$, we have 
\begin{align}\label{eq:bound_from_sesqui}
\emb_M^N(p,q) \geq \frac{pq}{\area(C)}
\end{align}
for all $N \in \Z_{\geq 0}$.
\end{propletbox}

Here $\area(C)$ denotes the symplectic area of $C$, which depends only on its homology class $[C] \in H_2(M)$.
Note that, after rescaling, \eqref{eq:bound_from_sesqui} can be written equivalently as $\emb_M^N(1,p/q) \geq \frac{p}{\area(C)}$.

The above proposition allows us to reduce Theorem~\ref{thmlet:A} to the following key construction of singular algebraic curves.
\begin{thmletbox}[\cite{sesqui}]\label{thmlet:C}
  There exists a rational index zero $(p,q)$-sesquicuspidal symplectic curve $C \subset \CP^2$ with coprime $p \geq q \geq 1$ if and only if 
  \begin{enumerate}[label=(\alph*)]
    \item $p/q = \frac{\fib_{2k+5}}{\fib_{2k+1}}$ for $k \in \Z_{\geq -1}$, or
    \item $p/q > \tau^4$. 
  \end{enumerate}
 \end{thmletbox} 
\begin{remark}
Note that our convention is such that the first two cases in Theorem~\ref{thmlet:C} are $p/q = 2/1$ and $p/q = 5/1$.
Strictly speaking the cases in Theorem~\ref{thmlet:C} with $q=1$ are trivial, since then a $(p,q)$ is simply a nonsingular point, but it is convenient to include them in the statement because they still give important symplectic embedding obstructions via Proposition~\ref{proplet:B}.
\end{remark}
\begin{exercise}\label{ex:outer_corners_from_curves}
Combining Proposition~\ref{proplet:B} with Theorem~\ref{thmlet:C} gives precisely the lower bounds on $\emb_{B^4}^N$ needed for Theorem~\ref{thmlet:A}.
\end{exercise}
Thus in the remaining lectures we will focus on the proofs of Proposition~\ref{proplet:B} and Theorem~\ref{thmlet:C}, albeit with various detours.

\subsection{Outline of the remaining lectures}\label{subsec:outline}

Here is the plan for the remaining lectures:
\begin{itemize}
  \item   In Lecture~\ref{sec:enumerative_geometry}, we discuss the enumerative geometry of sesquicuspidal curves, and we use this formalism to sketch a proof of Proposition~\ref{proplet:B}.
  \item In Lecture~\ref{sec:well_placed_curves}, we define well-placed curves and discuss their cluster symmetries, and we give one construction of the ``outer corner curves'' alluded to in \S\ref{subsec:fibonacci_staircase}.
  \item In Lecture~\ref{sec:scattering_diagrams}, we give a brief crash course in scattering diagrams and formulate the tropical vertex theorem proved in \cite{gross2010tropical}.
  \item Finally, in Lecture~\ref{sec:putting_together} we put all of this together to give the proof of Theorem~\ref{thmlet:C}.
 \end{itemize}

\begin{remark}
While much of the theory developed in this lectures can be extended in various directions, for instance by replacing $\CP^2$ with other del Pezzo surfaces, for simplicity we will mostly focus our exposition on the case of the complex projective plane. 
\end{remark}

\section{Enumerative geometry of sesquicuspidal curves}\label{sec:enumerative_geometry}

In this lecture, we first recall in \S\ref{subsec:corner_curves} a classification theorem for rational unicuspidal plane curves, and we introduce the inner and outer curves. Then in \S\ref{subsec:inflation} we discuss inflation along singular curves and how the inner corner curves produce symplectic embeddings. 
Finally, in \S\ref{subsec:tangency_constraints} we discuss a general enumerative framework for pseudoholomorphic curves with cusp singularities, and we use this to sketch the proof of a generalization of Proposition~\ref{proplet:B}.

\subsection{The inner and outer corner curves}\label{subsec:corner_curves}

Let us begin by recalling the following result as motivation.
\begin{theorembox}[\cite{fernandez2006classification}]\label{thm:bob_et_al}
 There exists a $(p,q)$-unicuspidal rational algebraic curve in $\CP^2$ of degree $d$ with $p > q\geq 2$ and $\gcd(p,q) = 1$ if and only if $(d,p,q)$ is one of the following:
 \begin{enumerate}[label=(\alph*)]
    \item $(p,q) = (d,d-1)$ for $d \geq 3$
    \item $(p,q) = (2d-1,d/2)$ for $d \geq 4$ even
    \item $(p,q) = (\fib_{2k+3}^2,\fib_{2k+1}^2)$ and $d = \fib_{2k+3}\fib_{2k+1}$ for $k \geq 1$
    \item $(p,q) = (\fib_{2k+5},\fib_{2k+1})$ and $d = \fib_{2k+3}$ for $k \geq 1$
    \item $(p,q) = (22,3)$ and $d = 8$
    \item $(p,q) = (43,6)$ and $d = 16$. 
  \end{enumerate} 
\end{theorembox}
Recall that the index of such a curve $C$ is given by 
\begin{align*}
\ind(C) = 2(c_1(C) - p - q) = 2(3d-p-q).
\end{align*}
\begin{exercise}
 The indices of the curves in Theorem~\ref{thm:bob_et_al} are as follows:
 \[
\begin{array}{c|c|c|c|c|c|c}
 & (a) & (b) & (c) & (d) & (e) & (f) \\
\hline
\text{index} & 2d + 2 & d + 2 & 2 & 0 & -2 & -2
\end{array}
\]
 
\end{exercise}

Observe that, in (c), $\frac{\fib_{2k+3}^2}{\fib_{2k+1}^2}$ is precisely the $x$-value of the $k$th inner corner point in the Fibonacci staircase from Theorem~\ref{thm:McSch}, so we will refer to these as \hl{inner corner curves}. Similarly, in (d), $\frac{\fib_{2k+5}}{\fib_{2k+1}}$ is precisely the $x$-value of the $k$th outer corner point, and we will refer to these as \hl{outer corner curves}.
Note also that the outer corner curves have index zero, so by Proposition~\ref{proplet:B} their existence gives
\begin{align*}
\emb_{B^4}^N\left(1,\frac{\fib_{2k+5}}{\fib_{2k+1}}\right) \geq \frac{\fib_{2k+5}}{\fib_{2k+3}},
\end{align*}
where $\tfrac{\fib_{2k+5}}{\fib_{2k+3}}$ is precisely the $y$-value of the $k$th outer corner point.

\begin{remark} Here are a few more remarks about the curve families (a),(b),(c),(d),(e),(f) in Theorem~\ref{thm:bob_et_al}:
\begin{itemize}
  \item family (a) is given by $\{ZY^{d-1} = X^d\}$
  \item family (b) is given by $\{(ZY-X^2)^{d/2} = XY^{d-1}\}$
  \item Miyanishi--Sugie \cite{Miyanishi-Sugie} or Kashiwara \cite{kashiwara_hiroko} construct a pencil with general member (c) and special member (d)
  \item Orevkov \cite{orevkov2002rational} constructs families (d),(e),(f) (we will explain a related construction in \S\ref{subsec:shift_reflection_symmetries} below)
  \item families (e),(f) are a bit ``surprising'' to a symplectic geometer in the sense that they should not persist under a generic perturbation of the almost complex structure.
\end{itemize}
  
\end{remark}

\subsection{Inflation along sesquicuspidal curves}\label{subsec:inflation}

The following proposition, which is essentially an application of the technique of \hl{symplectic inflation}, constructs four-dimensional symplectic ellipsoid embeddings from singular curves.
While it is not directly related to the proof of Theorem~\ref{thmlet:C}, it does help illustrate the role connection between sesquicuspidal curves and (stabilized) symplectic embeddings.

\begin{propletbox}\label{proplet:D}
Let $M$ be a closed symplectic four-manifold and $C \subset M$ a $(p,q)$-sesquicuspidal symplectic curve such that $[C] = c \pd([\om_M])$ for some $c \in \R_{>0}$ and $[C] \cdot [C] \geq pq$.
Then $\emb_M(p,q) \leq c$ (i.e. $\emb_M(1,p/q) \leq c/q$).
\end{propletbox}

\begin{remark}
 Note that we have
 \begin{align*}
  \vol(M) = \tfrac{1}{2} \int \om_M \wedge \om_M = \tfrac{1}{2} \pd([\om_M]) \cdot \pd([\om_M]) = \tfrac{1}{2c^2} [C] \cdot [C].
  \end{align*} 
In the case $[C] \cdot [C] = pq$, Proposition~\ref{proplet:D} corresponds to an embedding $E\left(\tfrac{p}{c},\tfrac{q}{c}\right) \hooksymp M$ (up to small error), which is a full filling since $\vol(E\left(\tfrac{p}{c},\tfrac{q}{c}\right)) = \tfrac{pq}{2c^2} = \vol(M)$.
\end{remark}

\begin{exercise}\label{ex:inner_corners_from_curves}
Show that the inner corner curves (i.e. family (c) from Theorem~\ref{thm:bob_et_al}) satisfy $[C]\cdot [C] = d^2 = pq$, and these translate into full fillings via Proposition~\ref{proplet:D}.
\end{exercise}

Combining Exercise~\ref{ex:outer_corners_from_curves} and Exercise~\ref{ex:inner_corners_from_curves}, we see that the entire (stabilized) Fibonacci staircase $\emb_{B^4}^N(1,a)$ for $a \in [1,\tau^4]$ and $N \in \Z_{\geq 0}$ follows from the existence of the outer and inner corner unicuspidal curves (i.e. families (c) and (d) in Theorem~\ref{thm:bob_et_al}).

\sss

Although Proposition~\ref{proplet:D} will suffice for our purposes here, let us also mention that symplectic inflation can also be performed along curves with more complicated singularities, for example a \hl{$k$-fold $(p,q)$-multicusp}, i.e. the singularity locally modeled on 
\[\left\{\prod_{j=1}^{k} (x^p - C_j y^p) = 0\right\} \subset \C^2\]
for some (pairwise distinct) constants $C_j \in \C$.
In fact, the following theorem shows that the existence of such a curve is {\em equivalent} to the existence of a corresponding ellipsoid embedding.

\begin{theorembox}[Ophstein \cite{opshtein2015symplectic}]
Let $M^4$ be a symplectic four-manifold with rational symplectic form class $[\om_M] \in H^2(M;\Q)$, and fix $\tau \in \Q_{>0}$ and  relatively prime $p,q \in \Z_{\geq 1}$. Then we have $E(\tau p,\tau q) \hooksymp M$ if and only if there exists an irreducible symplectic curve $C \subset M$ which has a $k\tau$-fold $(p,q)$-multicusp and satisfies $[C] = k\pd([\om_M])$ and $[C] \cdot [C] \geq k\tau p q$, for some $k \in \Z_{\geq 1}$.
\end{theorembox}

\sss

We end this subsection with a proof sketch of Proposition~\ref{proplet:D}, which introduces the method of symplectic inflation and also some useful pictures around resolutions of singularities.

\begin{proof}[Proof sketch of Proposition~\ref{proplet:D}] \hfill

\step{1}: We first resolve the cusp singularity of $C$ by a $(p,q)$-weighted symplectic blowup as in Figure~\ref{fig:res1}, giving a new curve $\wt{C} \subset \wt{M}$.
In more detail, this means that we start with a neighborhood of the cusp point $\pp$ which is symplectomorphic to a neighborhood of the form $\mu^{-1}(U) \subset \R^4$, where $U$ is a neighborhood of the origin on $\R^2_{\geq 0}$ and $\mu: \R^4 \ra \R^2_{\geq 0}$ is the standard moment map.
We may assume that this symplectomorphism identifies $C$ locally near $\pp$ with the pre-image under $\mu$ of a ray in $\R^2_{\geq 0}$ through the origin with direction $(p,q)$. We can assume that $\mu^{-1}(U)$ contains the ellipsoid $E(\eps p,\eps q)$ for some $\eps > 0$ sufficiently small.

The weighted blowup $\wt{M} := \bl_{p,q}M$ is then obtained by symplectically cutting out the ellipsoid $E(\eps p,\eps q)$, i.e. replacing the local toric picture with $2$ edges with one with $3$ edges as in the right side of Figure~\ref{fig:res1}.
Here $\wt{C} \subset \wt{M}$ is the strict transform of $C$, given by suitably truncating the line segment $\mu(C)$.
Note that $\wt{C}$ no longer has a $(p,q)$ cusp but may still have some additional double points, while $\wt{M}$ has (at most) two orbifold points, due to the two non-Delzant corners in the local toric picture.

\begin{figure}[H]
\centering
\includegraphics[width=1\textwidth]{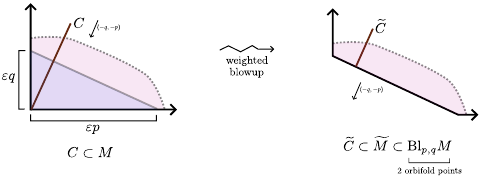}
\caption{The local toric picture for the $(p,q)$-weighted symplectic blowup $M \rightsquigarrow \wt{M}$.}
\label{fig:res1}
\end{figure}

\step{2} (optional): Although the orbifold points are not necessarily problematic, it is sometimes preferable to avoid them, which we can do by further resolving $\wt{M}$. In terms of the local toric picture, this corresponds to adding many small edges so that all of the vertices become Delzant -- see Figure~\ref{fig:res2}. We denote the resulting nonsingular symplectic manifold by $\wtt{M}$ and the new (essentially unaffected) curve by $\wtt{C} \subset \wtt{M}$.

\begin{figure}[H]
\centering
\includegraphics[width=1\textwidth]{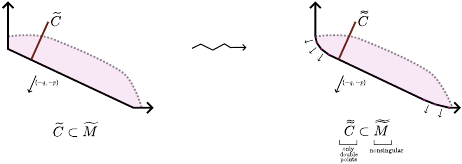}
\caption{The local toric picture for the resolution of orbifold singularities $\wt{M} \rightsquigarrow \wtt{M}$.}
\label{fig:res2}
\end{figure}

\begin{example}
Figure~\ref{fig:res3} shows a concrete example of the process of going from $\wt{M}$ to $\wtt{M}$ by resolving the two orbifold points.
In the local toric pictures, $\wt{M}$ has edges with outward normals $(-1,0),(-2,-3),(0,-1)$, while and $\wtt{M}$ has two new edges with outward normals $(-1,-1)$ and $(-1,-2)$. One can easily check that all the corners for $\wtt{M}$ are Delzant.

\begin{figure}[H]
\centering
\includegraphics[width=1\textwidth]{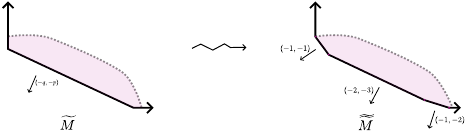}
\caption{The orbifold resolution process $\wt{M} \rightsquigarrow \wtt{M}$ in the case $(p,q) = (3,2)$.}
\label{fig:res3}
\end{figure}
\end{example}

\begin{remark}
 Note that the process of going from $C \subset M$ to $\wtt{C} \subset \wtt{M}$ is a symplectic version of embedded resolution of singularities for a $(p,q)$ cusp. The combinatorics of this resolution process are neatly encoded in the box diagram for $(p,q)$.
 For example, the box diagram for $(p,q) = (3,2)$ is shown in Figure~\ref{fig:box_diagram_32}, with the $3$ boxes corresponding to the $2$ small added edges in Figure~\ref{fig:res3} -- see \cite[\S4.1]{cusps_and_ellipsoids} for more details.

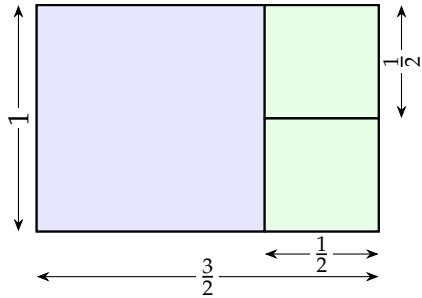
\begin{figure}[H]
\centering
\begin{tikzpicture}[scale=3,>=Stealth]

  \pgfmathsetmacro{\H}{1}        
  \pgfmathsetmacro{\W}{3/2}      
  \pgfmathsetmacro{\hf}{1/2}     

  \draw[thick] (0,0) rectangle ({\W},{\H});

  \draw[thick] (1,0) -- (1,{\H});

  \fill[blue!10] (0,0) rectangle (1,{\H});

  \draw[thick] (0,0) rectangle (1,{\H});

  \draw[thick] (1,{\hf}) -- ({\W},{\hf});   

  \fill[green!10] (1,0) rectangle ({\W},{\hf});
  \fill[green!10] (1,{\hf}) rectangle ({\W},{\H});

  \draw[thick] (1,0) rectangle ({\W},{\hf});
  \draw[thick] (1,{\hf}) rectangle ({\W},{\H});


  \draw[<->] (0,-0.20) -- ({\W},-0.20)
    node[midway,fill=white,inner sep=1pt] {$\frac{3}{2}$};

  \draw[<->] (-0.08,0) -- (-0.08,1)
    node[midway,fill=white,inner sep=1pt,rotate=90] {$1$};

  \draw[<->] (1,-0.10) -- ({1+\hf},-0.10)
    node[midway,fill=white,inner sep=1pt] {$\frac{1}{2}$};

  \draw[<->] ({\W+0.10},{\hf}) -- ({\W+0.10},{\H})
    node[midway,fill=white,inner sep=1pt,rotate=90] {$\frac{1}{2}$};

\end{tikzpicture}

\caption{The box diagram for $(p,q) = (3,2)$.}
\label{fig:box_diagram_32}
\end{figure}

\end{remark}

\end{proof}

\step{3}: We now smooth the remaining double points of $\wtt{C}$. This is straightforward in the symplectic category, and locally modeled on trading $\{xy = 0\}$ for $\{xy = \de\}$ in $\C^2$ for $\de > 0$ small. We denote the resulting curve by $\wtt{C}_\sm \subset \wtt{M}$.

\step{4}: We can find a closed two-form $\Om$ on $\wtt{M}$ which is supported in a small neighborhood of $\wtt{C}_\sm$ and satisfies $[\Om] = \pd(\wtt{C}_\sm)$. Here it is necessary that $\wtt{C}_\sm$ has nonnegative self-intersection number, which comes from our assumption $[C] \cdot [C] \geq pq$. Note that finding $\Om$ is straightforward in the special case that $\wtt{C}_\sm$ has zero self-intersection number, since then it has a trivial symplectic normal bundle and we can take $\Om$ to be simply a two-form pulled back from the normal direction.

\step{5}: We now ``inflate'' $\wtt{M}$ by replacing its symplectic form $\om_{\wtt{M}}$ with the new one 
$\om^s_{\wtt{M}} := \om_{\wtt{M}} + s\Om$ for $s \in \R_{>0}$. Let $E$ denote the symplectic divisor in $\wtt{M}$ corresponding to the edge in the local toric picture with outward normal direction $(-p,-q)$ as in Figure~\ref{fig:res4}. Note that $E$ has symplectic area approximately $\eps$ with respect to $\om_{\wtt{M}}$, given by the affine length of the corresponding edge in the local toric picture.
Then $E$ has area approximately $\eps + s$ with respect to the symplectic form $\om^s_{\wtt{M}}$.

\begin{figure}[H]
\centering
\includegraphics[width=.5\textwidth]{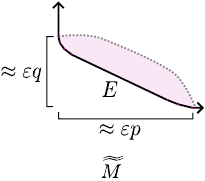}
\caption{The edge $E \subset \wtt{M}$ along which we blow down.}
\label{fig:res4}
\end{figure}

\step{6}: We now reverse the blowup process $M \rightsquigarrow \wtt{M}$ by blowing down $E$ and the small edges introduced in Step 2, which gives back $M$ but now with a symplectic form $\om_M^s$. 
This can be again understood as a straightforward local toric substitution,
but observe that the symplectic manifold $(M,\om_M^s)$ now naturally contains the ellipsoid $E([\eps+s]p,[\eps+s]q)$,
and we have $[\om_M^s] = [\om_M] + s\pd(C) = (1+sc)[\om_M^s]$ in $H^2(M;\R)$.
After rescaling by the factor $1+sc$ and applying Moser's stability theorem, we thus have a symplectic embedding
\begin{align*}
E\left(\tfrac{\eps+s}{1+sc}p,\tfrac{\eps+s}{1+sc}q\right) \hooksymp (M,\om_M).
\end{align*}
Taking $s \ra \infty$, we get an embedding
$E(\tfrac{p}{c'},\tfrac{q}{c'}) \hooksymp (M,\om_M)$ for $c'$ arbitrarily close to $c$.

\subsection{Multidirectional tangency constraints}\label{subsec:tangency_constraints}

Let $M^{2n}$ be a closed $2n$-dimensional symplectic manifold, and $\vecD = (D_1,\dots,D_n)$ be a collection of nonsingular local symplectic divisors which have a symplectic normal crossing intersection at a point $\pp \in M$ (that is, they look locally like $\{z_1\cdots z_n = 0\} \subset \C^n$ in symplectic Darboux coordinates).

\begin{notation}
Let $\calJ(M,\vecD)$ denote the set of tame almost complex structures on $M$ which are integrable near $\pp$ and preserve $J_1,\dots,J_n$.
\end{notation}
Put $z_0 := [0:1] \in \CP^1$.

\begin{definition}
  Given an almost complex structure $J \in \calJ(M,\vecD)$, a homology class $A \in H_2(M)$, and a tuple $\vecm = (m_1,\dots,m_n) \in \Z_{\geq 1}^n$, let $\calM^J_{M,A}\lll \calC_{\vecD}^{\vecm} \pp \rrr$ denote the moduli space of $J$-holomorphic maps $u: \CP^1 \ra M$ which satisfy $[u] = A$, $u(z_0) = \pp$, and such that $u$ has contact order at least $m_i$ with $D_i$ at $z_0$ for $i = 1,\dots,n$, modulo biholomorphic reparametrizations of $\CP^1$ fixing $z_0$.
\end{definition}

Here the contact order constraints $\lll \calC_{\vecD}^{\vecm}\pp\rrr$ with $D_1,\dots,D_n$ is called a \hl{multidirectional tangency constraint}.

\begin{remark}\label{rmk:multidir}\hfill
\begin{itemize}
  \item When $\vecm = (m,\underbrace{1,\dots,1}_{n-1})$, this reduces to the $m$-fold \hl{local tangency constraint} (with respect to $D := D_1$) studied by Cieliebak--Mohnke and others (see e.g. \cite{CM_SFT_compactness,CM2,tonk,McDuffSiegel_counting}).
  \item If $u: \CP^1 \ra M^4$ has contact order exactly $p$ with $D_1$ and exactly $q$ with $D_2$, then it necessarily has a $(p,q)$ cusp singularity at $\pp$ (assuming no other branches of $u$ meet $\pp$).
\end{itemize}
\end{remark}

The following theorem combines some results in \cite{cusps_and_ellipsoids}, which builds on some technical results in \cite{SDEP}.
\begin{thmletbox}\label{thmlet:E} Fix a semipositive symplectic manifold $M^{2n}$ and homology class $A \in H_2(M)$, and put $\vecm = (p,q,\underbrace{1,\dots,1}_{n-2})$ for some coprime $p,q \in \Z_{\geq 1}$ satisfying $p+q = c_1(A)$. 
\begin{enumerate}[label=(\roman*)]
  \item For generic $J \in \calJ(M,\vecD)$, the moduli space $\calM_{M,A}^J\lll \calC_{\vecD}^{\vecm}\pp\rrr$ is finite and regular, and the signed count $\#_{M,A}^{\vecm} := \# \calM_{M,A}^J\lll \calC_{\vecD}^{\vecm}\pp\rrr$ is independent of $\pp,\vecD,J$.
  \item For any other closed symplectic manifold $Q^{2N}$ such that $M^{2n} \times Q^{2N}$ is semipositive, we have
  \begin{align*}
  \#_{M\times Q,A \times [\op{pt}]}^{(\vecm,\vec{1})}= \#_{M,A}^{\vecm},
  \end{align*}
  where $(\vecm,\vec{1}) = (m_1,\dots,m_n,\underbrace{1,\dots,1}_N)$.
  \item If $\#_{M,A}^\vecm \neq 0$, then we have 
\begin{align*}
  \emb_M^N(p,q,a_3,\dots,a_n) \geq \frac{pq}{\area(A)}
\end{align*}
whenever $a_3,\dots,a_n > pq$.
\item If $2n=4$, then $\#_{M,A}^\vecm \in \Z_{\geq 0}$, and if there exists a rational $(p,q)$-sesquicuspidal symplectic curve $C \subset M$ with $[C] = A$, then $\#_{M,A}^\vecm > 0$.
\end{enumerate}
\end{thmletbox} 

Here semipositivity is a technical topological assumption which allows us to achieve generic transversality without virtual techniques. 
One could think of Theorem~\ref{thmlet:E} as a generalization of the approach to semipositive Gromov--Witten invariants taken in \cite{JHOL}.

\begin{proof}[Proof of Proposition~\ref{proplet:B}]
This follows directly from parts (iii) and (iv) of Theorem~\ref{thmlet:E}.
\end{proof}

\begin{remark}
 There is also a version of Theorem~\ref{thmlet:E} for more general $\vecm \in \Z^n_{\geq 1}$, but the combinatorial assumptions on $\vecm$ and $c_1(A)$ are more complicated to state -- see \cite{cusps_and_ellipsoids}. 
The special case of Theorem~\ref{thmlet:E} for local tangenc constraints was considered earlier in \cite{McDuffSiegel_counting}. 
\end{remark}

\begin{example}\label{ex:multidir1}
According to Remark~\ref{rmk:multidir}(i), $\#_{\CP^2,d\ell}^{(3d-1,1)}$ is the count of degree $d$ rational plane curves satisfying an order $3d-1$ local tangency constraint.
These were computed in \cite{McDuffSiegel_counting} as follows:
\[\begin{array}{|c|c|c|c|c|c|}
\hline
d & 1 & 2 & 3 & 4 & 5 \\ \hline
\#_{\CP^2,d\ell}^{(3d-1,1)} & 1 & 1 & 4 & 26 & 217 \\ \hline
\end{array}
\]
Note that this sequence appears as A364973 in the On-line Encyclopedia of Integer Sequences, where it is also seen to make important appearances in the mirror symmetry literature.
\end{example}
\begin{example}\label{ex:multidir2}
The counts corresponding to the outer corner curves in $\CP^2$ are all equal to $1$:
\[\begin{array}{|c|c|c|c|c|c|}
\hline
d & 1 & 2 & 3 & 4 & 5 \\ \hline
\#_{\CP^2,\fib_{2k+3}\ell}^{(\fib_{2k+5},\fib_{2k+1})} & 1 & 1 & 1 & 1 & 1 \\ \hline
\end{array}
\]
\end{example}
The general philosophy with Examples~\ref{ex:multidir1} and \ref{ex:multidir2} is that unicuspidal curves tend to be unique (after imposing a suitable multidirectional tangency constraint), whereas sesquicuspidal curve counts tend to grow rapidly.

\sss

We end this lecture by briefing giving some of the ideas underlying Theorem~\ref{thmlet:E}.

\begin{proof}[Proof ideas for Theorem~\ref{thmlet:E}]\hfill
\begin{enumerate}[label=(\roman*)]
  \item The standard approach is to use Gromov's compactness theorem, but here we need extra care to rule out additional ``bad degenerations''. For instance, a curve with a $(11,5)$ cusp could degenerate into two components, one with an $(8,3)$ cusp and one with a $(3,2)$ cusp, which meet at the singular point as in Figure~\ref{fig:cusp_deg}.

\begin{figure}[h]
\centering
\includegraphics[width=0.6\textwidth]{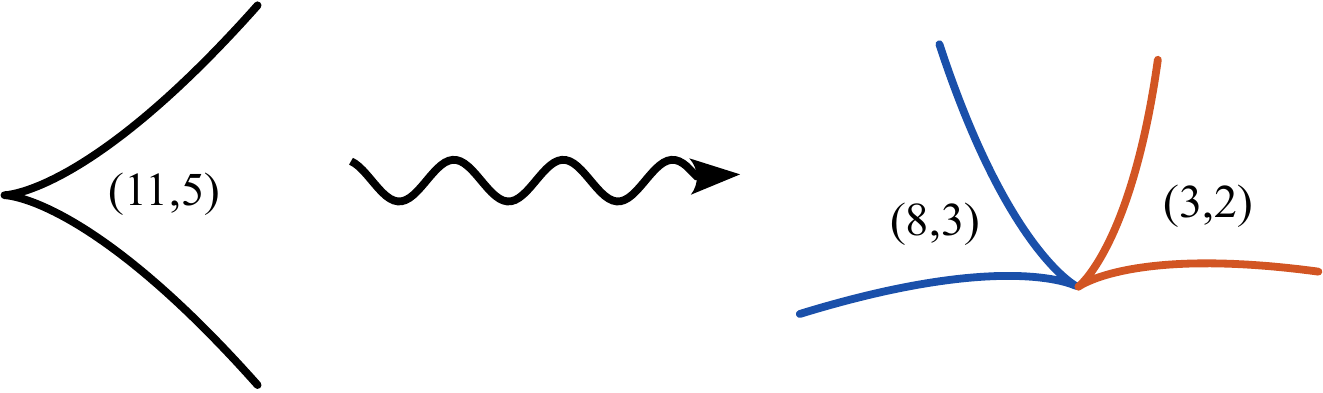}
\caption{Degeneration of a curve with a $(11,5)$ cusp}
\label{fig:cusp_deg}
\end{figure}

The following is our key tool for ruling out such degnerations.
\begin{lemletbox}[{``hidden constraints'' {\cite[\S3.2]{cusps_and_ellipsoids}}}] \label{lem:hid_con}
  
  If $\lll \calC^\vecm \rrr$ degenerates into $\lll \calC^{\vecm_1} \rrr,\dots,\lll \calC^{\vecm_k} \rrr$ then we must have
  \begin{align}\label{eq:hid_con}
  \sum\limits_{i=1}^k \min \{a_1m_1^i,\dots,a_nm_n^i\} \geq \min \{ a_1m_1,\dots,a_nm_n \}
  \end{align}
for any $\veca = (a_1,\dots,a_n) \in \R_{>0}^n$.
Here we put $\vecm = (m_1,\dots,m_n)$ and $\vecm_i = (m_1^i,\dots,m_n^i)$ for $i=1,\dots,k$,
\end{lemletbox}
For example, taking $\veca = (5,11)$ rules out the hypothetical degeneration depicted in Figure~\ref{fig:cusp_deg}.

The rough heuristic idea behind Lemma~\ref{lem:hid_con} is as follows. After surrounding the cusp by an ellipsoid $E := E(\eps \cdot \veca)$ for $\eps > 0$ sufficiently small and stretching the neck along $\bdy E$, the resulting degeration picture involves a rational punctured curve $C$ in the symplectization $\R \times \bdy E$ with one negative asymptotic Reeb orbit $\ga$ and many positive asymptotic Reeb orbits $\ga_1,\dots,\ga_k$. An explicit dictionary between cusp singularities and Reeb orbits in $\bdy E$ is given in \cite{SDEP}, and in particular nonnegativity of $C$ gives the action inequality $\sum\limits_{i=1}^k \calA_{\bdy E}(\ga_i) \geq \calA_{\bdy E}(\ga)$, which translates directly into ~\eqref{eq:hid_con}. 

\begin{remark}
The proof of Lemma~\ref{lem:hid_con} relies on neck stretching and the SFT compactness theorem. It would be interesting to see if there is a more algebraic or classical argument.
\end{remark}

  \item We argue using a product almost complex structure on $M \times Q$:
  \begin{itemize}
    \item using some elementary\underline{Fredholm theory}, one argues that regular curves $C \subset M$ remain regular as curves in $M \times Q$
    \item using some basic \underline{intersection theory}, one shows that all relevant curves in $M \times Q$ must be entirely contained in $M \times \{\pt\}$. 
  \end{itemize}
  Together, these essentially show that our curve counts are unaffected by the stabilization process.

\item We will assume $2n=4$ for simplicity. Suppose that we have an embedding $E\left(\frac{p}{c},\frac{q}{c}\right) \hooksymp M$. By assumption we have a curve $C$ with a $(p,q)$ cusp and $[C] = A$, which looks locally near the cusp like the left hand side of Figure~\ref{fig:res1} with respect to an ellipsoid embedding $E(\eps p, \eps q) \hooksymp M$. 
By Exercise~\ref{ex:symp_area}, the portion of the curve lying over the triangle with vertices $(0,0),(\eps p,0),(0,\eps q)$ has symplectic area $\eps pq$.
It is not difficult to show that we can find a family of ellipsoid embeddings $E(sp,sq) \hooksymp M$, $s \in [\eps,1/c]$, which interpolates between these two embeddings (using e.g. the ``extension after restriction principle'' -- see \cite[\S4.4]{Schlenk_old_and_new}).
By a compactness argument which extends (i) above, we can then find another curve $C'$ with the same local picture, but where the triangle now has vertices $(0,0),(p/c,0),(0,q/c)$. Since the portion of the curve lying over this triangle has symplectic area $\frac{pq}{c}$ (again using Exercise~\ref{ex:symp_area}), we must have $\area(C) > \frac{pq}{C}$.

\begin{exercise}\label{ex:symp_area}
Let $T$ be the triangle in $\R^2_{\geq 0}$ with vertices $(0,0),(rp,0),(0,rq)$ for some $r \in \R_{>0}$ and coprime $p,q \in \Z_{\geq 1}$, and let $\ga \subset \R_{\geq 0}^2$ be the shortest line segment from the origin to the hypotenuse of $T$. Let $C \subset \R^4$ be the visible symplectic curve whose moment map image is $\ga$, such that each fiber is a geodesic in the direction $(p,q)$ in the corresponding torus. Show that $\area(C) = rpq$. {\em Hint: see \cite[\S5.1a]{mcduff2024singular} for more on visible symplectic curves and their areas.}
\end{exercise}

\item This follows by a version of automatic transversality in dimension four, in the sense of \cite{hofer1997genericity,Wendl_aut}. More precisely, a well-known corollary of automatic transversality is that sufficiently nice curves count positively (c.f. \cite[\S5.2]{mcduff2021symplectic}).
\end{enumerate}

\end{proof}

\section{Well-placed curves and their symmetries}\label{sec:well_placed_curves}

In this lecture, we first recall the basics of Looijenga pairs in \S\ref{subsec:looijenga_pairs}, their toric models in \S\ref{subsec:toric_models}, and elementary transformations between toric models in \S\ref{subsec:elementary_transformations}.
We then define well-placed curves in \S\ref{subsec:well_placed_curves_def}, which on the one hand give sesquicuspidal curves and on the other hand are related to curves in nontoric blowups of toric surfaces.
We end by discussing cluster-type symmetries for well-placed curves in \S\ref{subsec:shift_reflection_symmetries}.

\subsection{Looijenga pairs}\label{subsec:looijenga_pairs}

\begin{definition} \hfill
\begin{itemize}
  \item A \hl{log Calabi--Yau pair} $(X,D)$ is a smooth complex projective variety $X$ with a reduced normal crossing divisor $D$ such that $[D] = \pd(c_1(X))$.
  \item A \hl{Looijenga pair} is a log Calabi--Yau pair $(X,D)$ such that $\dim_\C X = 2$ and $D$ is nodal with at least one node.
  \item A Looijenga pair $(X,D)$ is \hl{uninodal} if $D$ has exactly one node.
\end{itemize}
\end{definition}
\begin{exercise}
Show that a Looijenga pair $(X,D)$ being uninodal is equivalent to $D$ being irreducible. More precisely, if $D$ is a nodal anticanonical divisor in a smooth complex projective surface, then $D$ is irreducible if and only if it has either (a) geometric genus zero and exactly one node or (b) geometric genus one and no nodes. {\em Hint: apply the adjunction formula.}
\end{exercise}

\begin{example}\label{ex:CP2_N0}
  Letting $N \subset \CP^2$ denote a nodal elliptic curve, $(\CP^2,N)$ is a uninodal Looijenga pair.
  Up to projective equivalence, we may assume that $N$ is our preferred nodal cubic $N_0 = \{X^3 + Y^3 + XYZ = 0\}$.
  \end{example}
\begin{example}\label{ex:CP2_E}
  Letting $D \subset \CP^2$ denote a smooth elliptic curve, $(\CP^2,D)$ is a log Calabi--Yau pair but not a Looijenga pair.
\end{example}
\begin{example}\label{ex:CP2_3_lines}
 Letting $D \subset \CP^2$ denote the union of three lines in general position, $(\CP^2,D)$ is a Looijenga pair but not uninodal. 
\end{example}
\begin{example}\label{ex:CP2_edge_toric}
  Generalizing the above example, any toric pair $(X_\tor,D_\tor)$ is a Looijenga pair, where $X_\tor$ is a smooth complex toric surface and $D_\tor$ is its toric boundary divisor.
\end{example}

\begin{lemma} \hfill

\begin{enumerate}[label=(\arabic*)]
  \item\label{item:corner_blowup} For any Looijenga pair $(X,D)$, let $X'$ be the blowup at a node of $X$, and let $D'$ be the reduced total transform of $D$. Then $(X',D')$ is also a Looijenga pair.
  \item \label{item:edge_blowup}For any Looijenga pair $(X,D)$, let $X'$ be the blowup at a smooth point of $D$, and let $D'$ be the strict transform of $D$. Then $(X',D')$ is also a Looijenga pair.
\end{enumerate}
\end{lemma}
Here reduced total transform means we take the full preimage of $D$ under the blowup map $X' \ra X$, whereas strict transform means we throw out the exceptional divisor.

We will call the process in \ref{item:corner_blowup} a \hl{corner blowup} and the one in \ref{item:edge_blowup} an \hl{edge blowup} -- see Figure \ref{fig:corner_and_edge_blowups}.

\begin{figure}[htbp]
\centering
\includegraphics[width=0.8\textwidth]{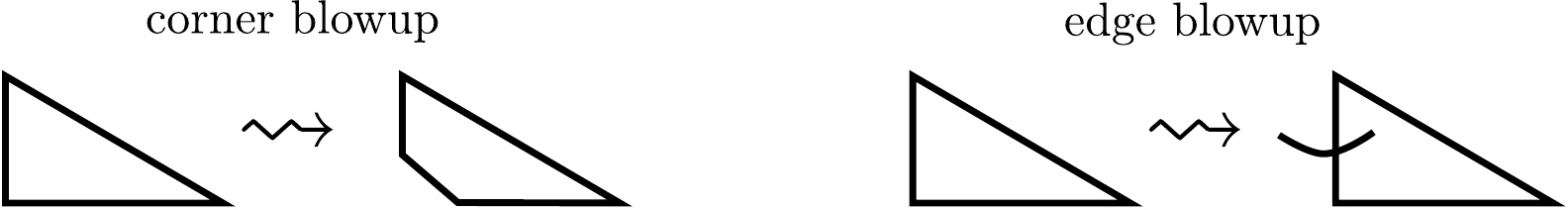}
\caption{Corner and edge blowups of Looijenga pairs.}
\label{fig:corner_and_edge_blowups}
\end{figure}

\begin{remark}
  Note that corner blowups of a Looijenga pair $(X,D)$ do not affect the \hl{interior} $X \setminus D$, but edge blowups do, with most of the new exceptional divisor lying in the new interior.
\end{remark}

\begin{terminology} 
We will call any iterated edge blowup of a toric pair an \hl{edge toric pair}.
\end{terminology}

\begin{remark} 
  Edge toric pairs can be thought of as a natural algebraic counterpart to symplectic almost toric fibrations.
  In the language of Gross--Hacking--Keel \cite{GHK_cluster}, Looijenga interiors $X \setminus D$ are precisely two-dimensional cluster varieties.

\end{remark}
We will see that the interior of every Looijenga pair can be identified with the interior of an edge toric pair.
In fact, there are usually many different such identifications, each called a {\em toric model}, and they are all related by a simple procedure called {\em elementary transformations}.

\subsection{Toric models}\label{subsec:toric_models}

\begin{definition}
  A \hl{toric model} $\calT$ for a Looijenga pair $(X,D)$ is:
\[
\begin{tikzcd}
(X_\tor,D_\tor) & (\wt{X},\wt{D}) \arrow[l, "\pi_e"'] \arrow[r, "\pi_c"] & (X,D),
\end{tikzcd}
\]
where \begin{itemize}
  \item $(X_\tor,D_\tor)$ is a toric pair
  \item $\pi_e$ is a composition of edge blowdowns
  \item $\pi_c$ is a composition of corner blowdowns.
\end{itemize}
\end{definition}
Informally, this says that $(X,D)$ becomes an edge toric pair after some corner blowups.

\begin{example}\label{ex:T0}
 Here is our most central example of a toric model for these lectures.
 Let $(\CP^2,N_0)$ be the uninodal Looijenga pair from Example~\ref{ex:CP2_N0}, where $N_0$ is a nodal cubic curve.
As illustrated in Figure~\ref{fig:toric_model_CP2}, after $3$ corner blowups this becomes identified with the edge blowup of $F_3$ at $2$ points.
Recall that the third Hirzebruch surface $F_3$ has moment polygon given by the quadrilateral depicted in Figure~\ref{fig:F3_edge}. Thus for this toric model we have:
\begin{itemize}
  \item  $X_\tor = F_3$
  \item $\pi_c$ is the composition of $3$ corner blowdowns
  \item $\pi_e$ is the composition of $2$ edge blowdowns.
\end{itemize}
We will denote this toric model by $\calT_0$.

\begin{figure}[H]
\centering
\includegraphics[width=1\textwidth]{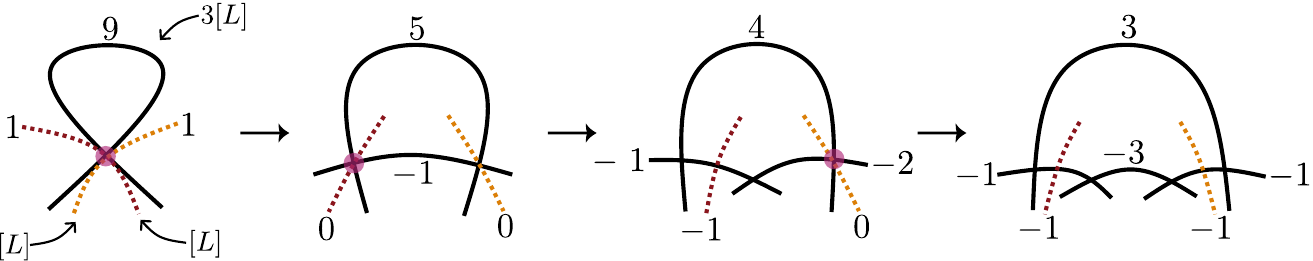}
\caption{The toric model $\calT_0$ for $(\CP^2,N_0)$, where $N_0$ is a nodal cubic. The divisor components are labeled by their self-intersection numbers, and the dashed lines denote the tangent lines to the two branches of $N_0$ at its double point, along with their subsequent strict transforms.}
\label{fig:toric_model_CP2}
\end{figure}

\begin{figure}[H]
\centering
\includegraphics[width=.4\textwidth]{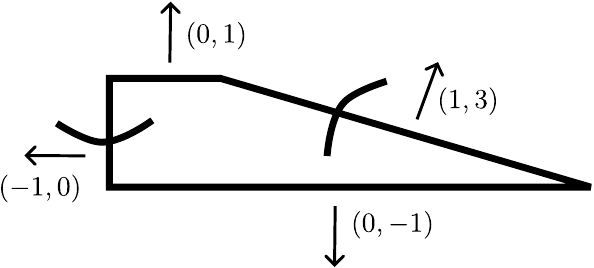}
\caption{The edge toric pair $(\wt{X},\wt{D})$ associated to the toric model $\calT_0$ for $(\CP^2,N_0$), where $\wt{X}$ is a two-point edge blowup of $F_3$.}
\label{fig:F3_edge}
\end{figure}
\end{example}

\begin{proposition}[{\cite[Prop. 1.3]{gross2015mirror}}]
Every Looijenga pair $(X,D)$ admits a toric model.
\end{proposition}

\begin{remark}
  Each toric model gives a cluster toric chart on the interior $X \setminus D$ via the identification 
  \[
\begin{tikzcd}
(\C^*)^2 \arrow[r, phantom, "\cong"] &
X_\tor \setminus D_\tor \arrow[r, hook, "\pi_e^{-1}"'] &
\wt{X} \setminus \wt{D} \arrow[r, phantom, "\cong"] &
X \setminus D.
\end{tikzcd}
\]

\end{remark}

\subsection{Elementary transformations}\label{subsec:elementary_transformations}

We begin with a simple example to motivate elementary transformations.
Put $F_0 := \CP^1 \times \CP^1$, and let $\pi: F_0 \ra \CP^1$ be the projection onto the first factor.
Pick a point $\pp \in F_0$, and let $F = \pi^{-1}(\pi(\pp))$ be the fiber of $\pi$ through $\pp$.
Now blow up $F_0$ at $\pp$, and then blow down the resulting strict transform $\wt{F}$ of $F$ (note that this is a $(-1)$-curve).
The result is the first Hirzebruch surface $F_1$. Note that $F_0$ and $F_1$ are both toric, and the fibration $\pi$ induces the usual fibration $F_1 \ra \CP^1$ as a ruled surface. 
See Figure~\ref{fig:F0_elem_trans_F1}.

\begin{figure}[h]
\centering
\includegraphics[width=1\textwidth]{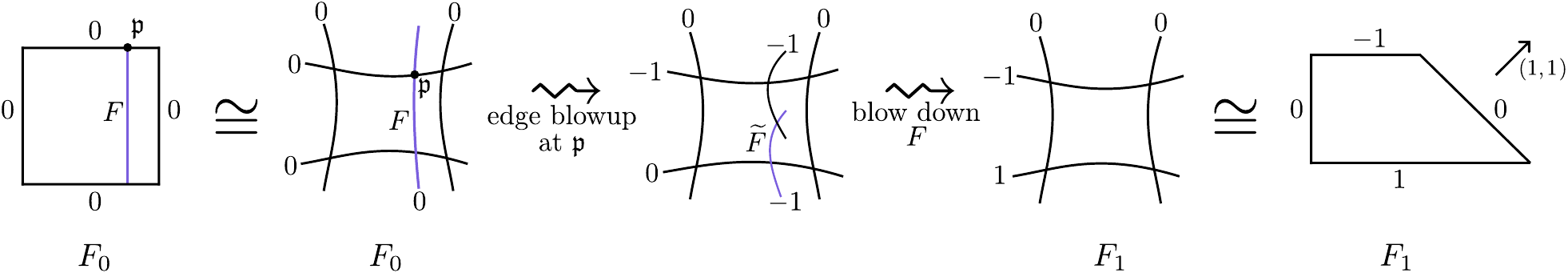}
\caption{Passing from $F_0$ to $F_1$ by an elementary transformation. Here the numbers denote self-intersection numbers and we freely switch between toric moment polygon representations and cartoon pictures depicting the relevant curves.}
\label{fig:F0_elem_trans_F1}
\end{figure}

More generally, let $X_\tor$ be a toric surface with primitive fan ray generators $v_1,\dots,v_k \in \Z^2$ (i.e. in the moment polygon picture these are the primitive inward normal vectors to the edges), with corresponding toric divisors $D_1,\dots,D_k \in X_\tor$.
Assume that we have $v_i = -v_j$ for some $1 \leq i < j \leq n$.
In this situation, toric algebraic geometry guarantees a projection map $\pi: X_\tor \ra \CP^1$ (this is an example of a toric morphism -- see e.g. \cite[\S3.3]{cox2011toric}).
Pick an edge point $\pp \in D_i$, and let $F := \pi^{-1}(\pi(\pp))$ be the fiber through $\pp$.
Let $\wt{X}_\tor$ denote the edge blowup of $X_\tor$ at $\pp$, and let $\wt{F} \subset \wt{X}_\tor$ be the strict transform of $F$. 
Let $X_\tor'$ denote the blowdown of $\wt{X}_\tor$ along $\wt{F}$.
\begin{lemma}
In this situation, $X_\tor'$ is again a toric surface, with primitive ray generators $v_1',\dots,v_k'$, where $v_s' := v_s + \max(v_s \wedge v_i,0)v_i$ for $s = 1,\dots,k$.
\end{lemma}
Here $v \wedge w$ denotes the determinant of the matrix with columns $v,w$, and the map $\R^2 \ra \R^2$, $w \mapsto w + \max(w\wedge v,0)v$ is called the \hl{primitive half-shear along $v$}.
\begin{example}
The primitive half-shear along $v = (1,0)$ is the map $\R^2 \ra \R^2$ given by
\begin{align*}
(x,y) = \begin{cases}
(x,y) & y \geq 0 \\
(x - y,y) & y < 0.
\end{cases}
\end{align*}
See Figure~\ref{fig:PL_shear_image} for an illustration.
\begin{figure}[h]
\centering
\includegraphics[width=.75\textwidth]{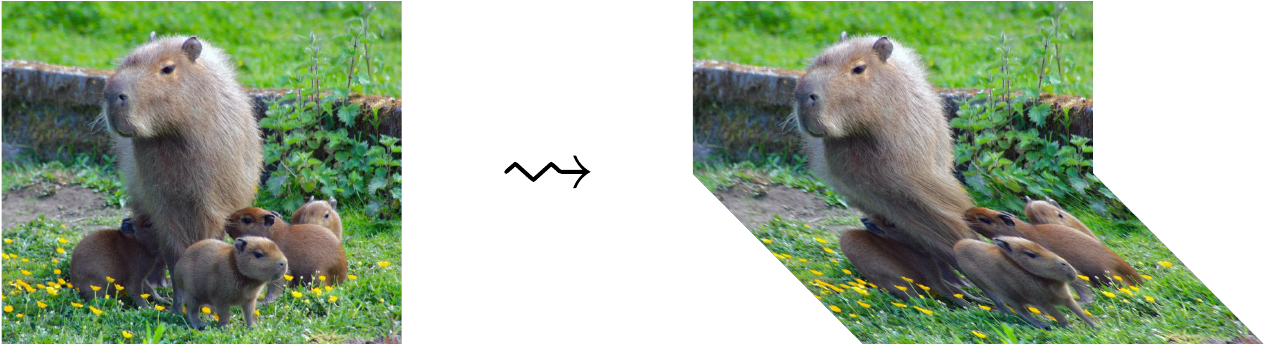}
\caption{The effect of the primitive half-shear along $v = (1,0)$.}
\label{fig:PL_shear_image}
\end{figure}
\end{example}
\begin{terminology}
The above procedure taking $X_\tor$ to $X_\tor'$ is called an \hl{elementary transformation}.
\end{terminology}

Now let $\calT$ be a toric model
\begin{tikzcd}
(X_\tor,D_\tor) & (\wt{X},\wt{D}) \arrow[l, "\pi_e"'] \arrow[r, "\pi_c"] & (X,D),
\end{tikzcd}
let $\pp_1,\dots,\pp_\ell \in D_\tor$ be the edge blowup points, and let $\mm_1,\dots,\mm_{\ell} \in \Z^2$ be the corresponding primitive fan ray generators.
Note that there may be repetitions in the list $\mm_1,\dots,\mm_{\ell}$, and not every primitive fan ray generator for $X_\tor$ is necessarily one of $\mm_1,\dots,\mm_{\ell}$ (that is, not every component of $D_\tor$ necessarily contains an edge blowup point).
Fix $1 \leq i \leq \ell$, and assume that $-\mm_i$ is a primitive fan ray generator for $X_\tor$.
Performing the elementary transformation of $X_\tor$ at $\pp_i$ gives a new toric model $\calT' = \mut_i\calT$ of the form
\[\begin{tikzcd}
(X_\tor',D_\tor') & (\wt{X},\wt{D}) \arrow[l, "\pi_e"'] \arrow[r, "\pi_c"] & (X,D)
\end{tikzcd}
\]
with associated data $\mm_1',\dots,\mm_\ell'$, where
\begin{align*}
\mm_s' = \begin{cases}
  -\mm_i & s = i\\
  m_s + \max(m_s \wedge m_i,0)m_i & s \neq i.
\end{cases}
\end{align*}

\begin{remark}
Recall that the interior $X \setminus D$ of a Looijenga pair $(X,D)$ does not change under corner blowups. For this reason, we will tend to view corner blowups as rather innocuous. In particular, for any $\mm_i$, we can always arrange that $-\mm_i$ appears as a primitive fan ray generator of $X_\tor$ after performing extra corner blowups, which we will often do implicitly. 
The most important data of a toric model is the vectors $\mm_1,\dots,\mm_\ell \in \Z^2$ which encode the edge blowups.
\end{remark}

Given a Looijenga pair $(X,D)$, we can consider its \hl{mutation graph}, whose vertices are toric models for $(X,D)$ and whose edges represent elementary transformations between toric models.
Pleasantly, it turns out that this graph is always connected.
\begin{proposition}[{Hacking--Keating \cite[Prop. 3.27]{hacking_keating_2022}}]
Given a Looijenga pair $(X,D)$, any two toric models are related by a sequence of corner blowups and elementary transformations.
\end{proposition}

\begin{example}\label{ex:mut_graph_T0}
Here is part of the mutation graph for the Looijenga pair $(\CP^2,N_0)$ from Example~\ref{ex:CP2_N0}, starting with its toric model $\calT_0$ from Example~\ref{ex:T0}.
\begin{center}
\[
\begin{tikzcd}[row sep=0.8em, column sep=0.7em, font=\small]
  & & & \begin{array}{c} \calT_0 \\ \mm_1 = (-1,-3) \\ \mm_2 = (1,0) \end{array} \arrow[dl, squiggly] \arrow[dr, squiggly] & & & \\
  & & \begin{array}{c} \mut_2\calT_0 \\ \mm_1 = (2,-3) \\ \mm_2 = (-1,0) \end{array} \arrow[dl, squiggly] & & \begin{array}{c} \mut_1 \calT_0 \\ \mm_1 = (1,3) \\ \mm_2 = (1,0) \end{array} \arrow[dr, squiggly] & & \\
  & \begin{array}{c} \mut_1\mut_2\calT_0 \\ \mm_1 = (-2,3) \\ \mm_2 = (5,-9) \end{array} \arrow[dl, squiggly] & & & & \begin{array}{c} \mut_2\mut_1\calT_0 \\ \mm_1 = (1,3) \\ \mm_2 = (-1,0) \end{array} \arrow[dr, squiggly] & \\
  \cdots & & & & & & \cdots
\end{tikzcd}
\]
\end{center}
\end{example}
 \begin{exercise}\label{ex:symmetries_for_T0}
In Example~\ref{ex:mut_graph_T0}, show that there is a matrix in $\gl(2,\Z)$ taking the $\mm_1,\mm_2$ data of $\calT_0$ to that of $\mut_2\mut_1\calT_0$ or $\mut_1\mut_2\calT_0$, but not to that of $\mut_1\calT_0$ or $\mut_2\calT_0$.
\end{exercise}

\subsection{Well-placed curves}\label{subsec:well_placed_curves_def}

Now let $(X,D)$ be a uninodal Looijenga pair, and let $\calB_\pm$ denote the two local branches of $D$ near its double point $\pp$.

\begin{definition}\label{def:wp}
 A curve $C \subset X$ is \hl{$(p,q)$-well-placed} (with respect to $D$) if:
 \begin{itemize}
    \item $C \cap D = \{\pp\}$, intersecting in a single branch of $C$
    \item the local intersection multiplicities are $(C \cdot \calB_-)_\pp = p$ and $(C \cdot \calB_+)_\pp = q$.
  \end{itemize} 
\end{definition}

\begin{convention}
It will be convenient to slightly extend Definition~\ref{def:wp} by saying that $C$ is $(p,0)$-well-placed or $(0,p)$-well-placed if it intersects $D$ in a single smooth point of $D$, with multiplicity $p$.
\end{convention}
\begin{remark}
Strictly speaking we should write $(D;\calB_-,\calB_+)$ instead of $D$ in Definition~\ref{def:wp} to emphasize the labeling of the two branches, but this will often be immaterial and thus suppressed from the notation.
\end{remark}

A simple but key observation is that, for $p,q$ coprime, a well-placed curve $C$ automatically has a $(p,q)$ cusp at the double point $\pp$ of $D$.
Moreover, we have $c_1(C) = C \cdot D = p +q$, and hence the index of zero is zero.
See Figure~\ref{fig:well_placed} for an illustration.

\begin{figure}[H]
\centering
\includegraphics[width=0.23\textwidth]{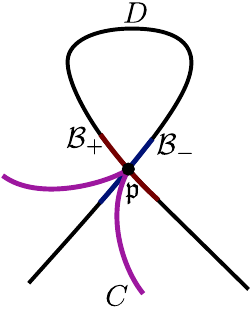}
\caption{A $(p,q)$-well-placed curve with respect to a uninodal anticanonical divisor $D$.}
\label{fig:well_placed}
\end{figure}

Our approach to Theorem~\ref{thmlet:C} will be to construct rational algebraic curves in $\CP^2$ which are $(p,q)$-well-placed with respect to the fixed nodal cubic $N_0$. By the above observation, these will in particular be $(p,q)$-sesquicuspidal of indx zero, but a priori this makes the search more difficult. Indeed, there are typically fewer well-placed curves then sesquicuspidal curves, as the following example illustrates.
\begin{example}
 One can show that the number of $(8,1)$-well-placed rational curves in $\CP^2$ with respect to the nodal cubic $N_0$ is $3$ (see \cite[\S7.1]{mcduff2024singular}), whereas according to Example~\ref{ex:multidir1} we have $\#_{\CP^2,3\ell}^{(8,1)} = 4$.
The discrepancy roughly comes from the fact that an almost complex structure preserving $N_0$ is not generic in the sense of Theorem~\ref{thmlet:E}(i), and heuristically the ``missing curve'' is $N_0$ itself.
\end{example}

\begin{propletbox}[``fundamental bijection'']\label{proplet:fund_bij}
 Let $(X,D)$ be a uninodal Looijenga pair with $\calT$ a toric model $\begin{tikzcd}
(X_\tor,D_\tor) & (\wt{X},\wt{D}) \arrow[l, "\pi_e"'] \arrow[r, "\pi_c"] & (X,D).
\end{tikzcd}$
There is a piecewise linear map $W_\calT: \Z^2_{\geq 0} \ra \Z^2$ and, for each $p,q \in \Z_{\geq 1}$, a bijection betweeen:
\begin{itemize}
  \item rational algebraic curves in $X$ which are $(p,q)$-well-placed with respect to $D$
  \item rational algebraic curves in $\wt{X}$ which intersect $\wt{D}$ in a single point, on the toric divisor component with primitive fan ray generator $W(p,q)$, with intersection multiplicity $\gcd(p,q)$.
\end{itemize}
\end{propletbox}
Here $W$ itself descends to a bijection from $\Z_{\geq 0}^2 / \sim$ to $\Z^2$ after identifying $(p,0) \sim (0,p)$ for all $p \in \Z_{\geq 1}$.
See Figure~\ref{fig:fund_bij} for an illustration.

\begin{figure}[H]
\centering
\includegraphics[width=1\textwidth]{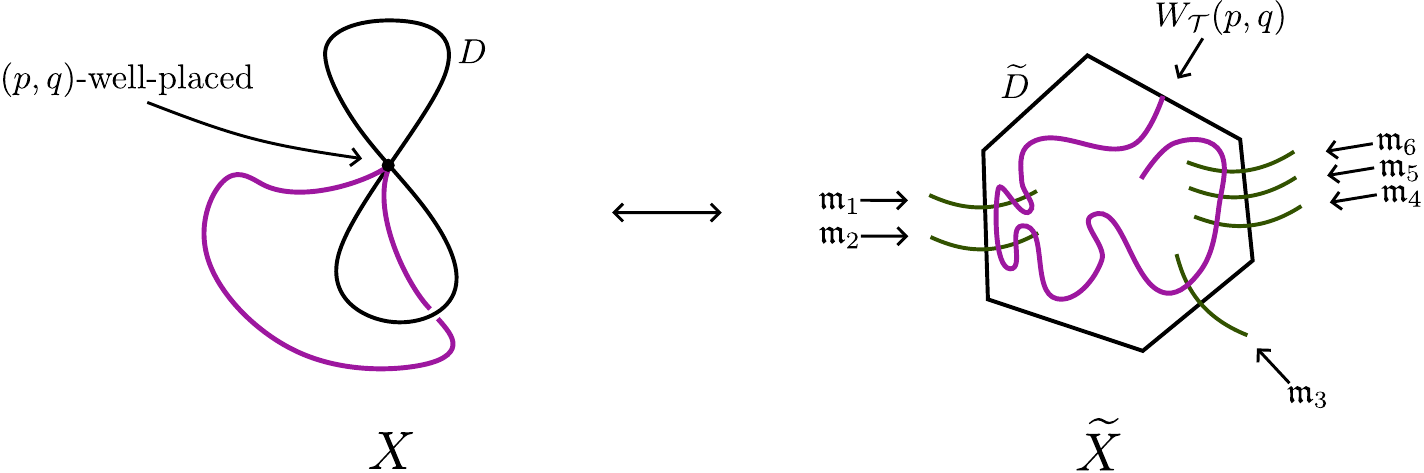}
\caption{The fundamental bijection relates well-placed curves in the uninodal Looijenga pair $(X,D)$ to curves in the edge toric pair $(\wt{X},\wt{D})$ which intersect the anticanonical divisor in one point.}
\label{fig:fund_bij}
\end{figure}

\begin{remark}
The curves in the second bullet of Proposition~\ref{proplet:fund_bij} could alternatively be thought of as complex affine lines in the interior of $(\wt{X},\wt{D})$, i.e. maps $\C \ra \wt{X} \setminus \wt{D}$, with specified asymptotic behavior.  
\end{remark}

\sss In the situation of Proposition~\ref{proplet:fund_bij}, let $\pp_1,\dots,\pp_\ell \in D_\tor$ be the edge blowup points, let $\mm_1,\dots,\mm_\ell \in \Z^2$ be the corresponding primitive fan ray generators, and let $E_1,\dots,E_\ell \subset \wt{X}$ be the corresponding exceptional curves (i.e. $E_s := \pi_e^{-1}(\pp_s)$ for $s = 1,\dots,\ell$).
Then, under the fundamental bijection, these correspond to curves $F_1,\dots,F_\ell \subset X$, where $F_s$ is $W_\calT^{-1}(\mm_s)$-well-placed for $s = 1,\dots,\ell$.
In particular, putting $u_s := W_\calT^{-1}(\mm_s)$ as a shorthand, this means that $F_s$ is a $u_s$-unicuspidal rational curve of index zero.

\begin{upshot}\label{upshot:curves}
 By iteratively mutating $\calT$ and looking at the corresponding curves $F_1,\dots,F_\ell$, we get lots of index zero rational unicuspidal curves in $X$.
\end{upshot}

\begin{example}\label{ex:T0_with_u1_u2}
  Revisiting Example~\ref{ex:mut_graph_T0} from the point of view of Upshot~\ref{upshot:curves}, the data $u_1,u_2$ for the corresponding curves is as follows:
\begin{center}
\[
\begin{tikzcd}[row sep=0.8em, column sep=0.7em, font=\small]
  & & & \begin{array}{c} \calT_0 \\ u_1 = (2,1) \\ u_2 = (1,2) \end{array} \arrow[dl, squiggly] \arrow[dr, squiggly] & & & &\\
  & & \cdots & & \begin{array}{c} \mut_1 \calT_0 \\ u_1 = (1,5) \\ u_2 = (1,2) \end{array} \arrow[dr, squiggly] & & &\\
  & & & & & \begin{array}{c} \mut_2\mut_1\calT_0 \\ u_1 = (1,5) \\ u_2 = (2,13) \end{array} \arrow[dr, squiggly] \arrow[dr, squiggly] & &\\
  & & & & & & \begin{array}{c} \mut_1\mut_2\mut_1\calT_0 \\ u_1 = (5,34) \\ u_2 = (2,13) \end{array} \arrow[dr,squiggly] &\\
  &&&&&&& \cdots
\end{tikzcd}
\]
\end{center}
Comparing these numerics with Theorem~\ref{thm:bob_et_al}(d), we see that these give all of the outer corner curves in $\CP^2$.
\end{example}

\begin{example}\label{ex:W_for_T0}
 For our main uninodal Looijenga pair $(\CP^2,N_0)$ with its toric model $\calT_0$ from Example~\ref{ex:T0},
 the piecewise linear map $W_{\calT_0}: \Z_{\geq 0}^2 / \sim \rightarrow \Z^2$ appearing in the fundamental bijection is given\footnote{As a small word of caution, our sign convention here for $W_{\calT_0}$ is the opposite of the one appearing in \cite{sesqui} (c.f. \cite[Table 4.3.1]{sesqui}.)} by:
\begin{align*}
  W_{\calT_0}(p,q) = 
  \begin{cases}
    (-q,p-5q) & p/q \geq 2 \\
    (q-p,q-2p) & 1/2 \leq p/q \leq 2 \\
    (p,q-2p) & p/q \leq 1/2.
  \end{cases}
\end{align*}
Recall that we have $\mm_1 = (-1,-3), \mm_2 = (1,0)$ and $u_1 = (2,1), u_2 = (1,2)$.
See Figure~\ref{fig:fund_bij_W} for an illustration.
\end{example}

\begin{figure}[H]
\centering
\includegraphics[width=0.9\textwidth]{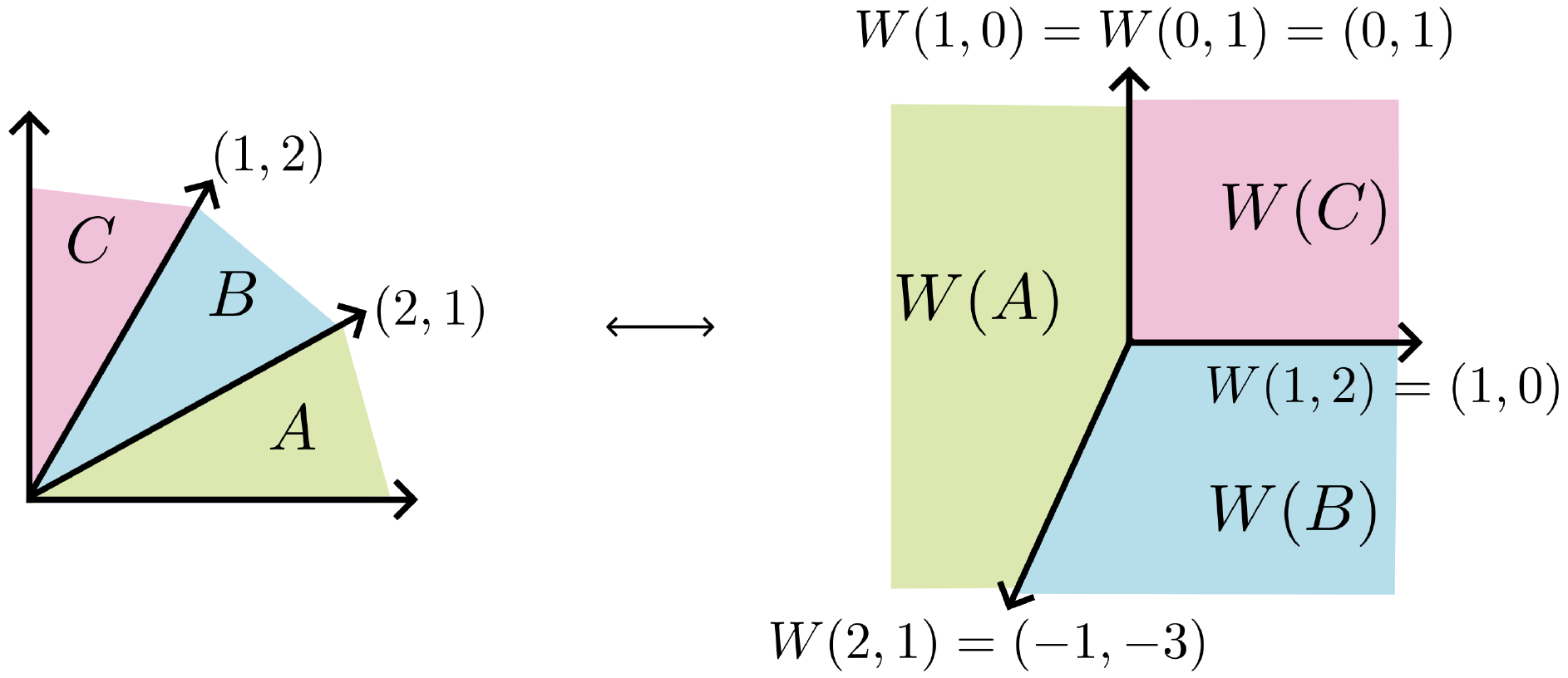}
\caption{The piecewise linear map $W_{\calT_0}: \Z^2_{\geq 0} \ra \Z^2$ appearing in the fundamental bijection associated to the toric model $\calT_0$ for $(\CP^2,N_0)$.}
\label{fig:fund_bij_W}
\end{figure}

\subsection{The shift and reflection symmetries}\label{subsec:shift_reflection_symmetries}

We now discuss symmetries for well-placed curves, focusing especially on our main example $(\CP^2,N_0)$ as in Example~\ref{ex:CP2_N0}.
Our treatment here is based on elementary transformations as in \S\ref{subsec:elementary_transformations} and is closely related to the ``generalized Orevkov twist'' from \cite{mcduff2024singular}, which is based on the construction in \cite{orevkov2002rational} (see also \cite{Kol} for an alternative perspective using Geiser-type involutions and \cite[\S3]{sesqui} for a more detailed summary). 
While our discussion here is somewhat abstract compared with the ones appearing in the preceding references, the one here embeds the symmetries into a more general context related to cluster transformations.

Suppose that a Looijenga pair $(X,D)$ has two toric models
\[
\calT: 
\begin{tikzcd}
(X_\tor,D_\tor) & (\wt{X},\wt{D}) \arrow[l, "\pi_e"'] \arrow[r, "\pi_c"] & (X,D)
\end{tikzcd}
\]
with data $\mm_1,\dots, \mm_\ell, u_1,\dots,u_\ell$,
and 
\[
\calT': 
\begin{tikzcd}
(X_\tor',D_\tor') & (\wt{X}',\wt{D}') \arrow[l, "\pi_e'"'] \arrow[r, "\pi_c'"] & (X,D)
\end{tikzcd}
\]
with data $\mm_1',\dots,\mm_\ell',u_1',\dots,u_\ell'$, such that there is a transformation $A \in \gl(2,\Z)$ satisfying 
$A \cdot \mm_s = \mm_s'$ for $s = 1,\dots,\ell$.
Then, up to corner blowups, we have isomorphisms
\[
\begin{tikzcd}
(X_\tor,D_\tor) \arrow[d, "\cong"'] 
  & (\wt{X},\wt{D}) \arrow[l, "\pi_e"'] \arrow[r, "\pi_c"] \arrow[d, "\cong"] 
  & (X,D) \\
(X_\tor',D_\tor') 
  & (\wt{X}',\wt{D}') \arrow[l, "\pi_e'"'] \arrow[r, "\pi_c'"] 
  & (X,D).
\arrow[from=1-1, to=2-2, phantom, "\circlearrowleft" description, pos=0.5]
\end{tikzcd}
\]
Recall that restricting to interiors gives isomorphisms:
$\pi_c: \begin{tikzcd}
\wt{X} \setminus \wt{D} \arrow[r, "\cong"] & X \setminus D
\end{tikzcd}$
and
$\pi_c': \begin{tikzcd}
\wt{X}' \setminus \wt{D}' \arrow[r, "\cong"] & X \setminus D.
\end{tikzcd}$
Using these together with the isomorphism $(\wt{X},\wt{D}) \cong (\wt{X}',\wt{D}')$, we get an isomorphism 
$\begin{tikzcd}
 X \setminus D \arrow[r, "\cong"] & X \setminus D
\end{tikzcd}$ which sends a $u_s$-well-placed curve to a $u_s'$-well-placed curve for $s = 1,\dots,\ell$.

In the case of $(\CP^2,N_0)$ with its toric model $\calT_0$, recall from Exercise~\ref{ex:symmetries_for_T0} that we have isomorphisms as above between $\calT_0$ and $\mut_2\mut_1\calT_0$, and also between $\calT_0$ and $\mut_1\mut_2\calT_0$.
According to the preceding discussion, these induces isomorphisms
\[
\Phi,\Psi: 
\begin{tikzcd}
\CP^2 \setminus N_0 \arrow[r, "\cong"] & \CP^2 \setminus N_0.  
\end{tikzcd}
\]

\begin{lemma}
 If $C \subset \CP^2$ is $(p,q)$-well-placed with respect to $N_0$, then $\Phi_*(C) \subset \CP^2$ is $(p',q')$-well-placed with respect to $N_0$, where
 \begin{align*}
 (p',q') = 
 \begin{cases}
(p,7p-q) & p/q \geq 1/7\\
(q-7p,p+7(q-7p)) & p/q < 1/7.
 \end{cases}
 \end{align*}
 Similarly, $\Psi_*(C)$ is $(p'',q'')$-well-placed with respect to $N_0$, where
\begin{align*}
 (p'',q'') = 
 \begin{cases}
 (q + 7(p-7q),p-7q)  & p/q \geq 7\\
(7q-p,q)  & p/q < 7.
 \end{cases}
 \end{align*}
\end{lemma}

\begin{notation}\hfill
  \begin{itemize}
    \item Let $\calS_{\CP^2}$ denote the set of rational numbers $p/q \geq 1$ for which there exists an index zero rational algebraic $(p,q)$-sesquicuspidal curve in $\CP^2$.
    \item Let $\calS_{\CP^2}^{N_0}$ denote the set of rational numbers $p/q \geq 1$ for which there exists a a rational algebraic curve in $\CP^2$ which is $(p,q)$-well-placed with respect to $(N;\calB_-,\calB_+)$ or $(N;\calB_+,\calB_-)$.
  \end{itemize}
\end{notation}
\NI   Note that we have the inclusion $\calS_{\CP^2}^{N_0} \subset \calS_{\CP^2}$.

\begin{corollary}\label{cor:shift_and_reflection} \hfill
\begin{itemize}
  \item The set $\calS_{\CP^2}^{N_0}$ is invariant under the ``shift'' symmetry $S(x) := 7 - \frac{1}{x}$
  \item The set $\calS_{\CP^2}^{N_0} \cap (7,\infty)$ is invariant under the ``reflection'' symmetry $R(x) := 7 + \frac{1}{x-7}$.
\end{itemize}
\end{corollary}
The terminology in Corollary~\ref{subsec:shift_reflection_symmetries} comes from the following.
\begin{exercise} The shift symmetry gives the disjoint union decomposition $(\tau^4,\infty) = \bigsqcup\limits_{i=0}^\infty S^i([7,\infty))$, and $R$ is an involution swapping $(7,8)$ and $(8,\infty)$. 
\end{exercise}

\begin{remark}
In light of Example~\ref{ex:T0_with_u1_u2}, it follows that the outer corner curves in $\CP^2$ can be constructed by repeatedly applying $\Phi$ to just two inital ``seed curves'', namely a $(2,1)$-well-placed line and a $(1,2)$-well-placed line.
\end{remark}

\begin{remark}
 If $C$ is a degree $d$ curve in $\CP^2$ which is $(3d,0)$-well-placed with respect to $N_0$ (e.g. $C$ is a flex line in the case $d=1$), then $\Phi_*(C)$ is $(21d,3d)$-well-placed but has a $(21d+1,3d)$ cusp. Note that the index of $\Phi_*(C)$ is $-2$.
 For $d=1,2$, this gives precisely the curves (e),(f) in Theorem~\ref{thm:bob_et_al},
 while for $d \geq 3$ we get curves which are sesquicuspidal but not unicuspidal.
\end{remark}

\section{Scattering diagrams and the tropical vertex}\label{sec:scattering_diagrams}

This lecture is essentially a crash course on scattering diagrams.
We first recall some basic definitions and the Kontsevich--Soibelman algorithm in \S\ref{subsec:scattering_basics}. 
Then, in \S\ref{subsec:toy_examples} we give three main examples of scattering diagrams and their Kontsevich--Soibelman completions, the last of which is really a theorem of Gross--Pandharipande and is central to our proof of Theorem~\ref{thmlet:C}.
Finally, in \S\ref{subsec:tropical_vertex} we give a (somewhat informal) statement of the tropical vertex theorem of Gross--Pandharipande--Siebert and connect it with the content of Lecture~\ref{sec:well_placed_curves}.

\subsection{Scattering diagram basics}\label{subsec:scattering_basics}

Put $A := \C[x,x^{-1},y,y^{-1}] \llbracket t \rrbracket$, and let $\aut(A)$ denote the group of automorphisms of $A$ as a $\C\llbracket t\rrbracket$-algebra. Given a lattice point $\mm = (a,b) \in \Z^2$, we use the shorthand $z^\mm$ for the corresponding monomial $x^ay^b$.

\begin{definition}
  A \hl{wall} in $\R^2$ is a pair $(\frakd,f)$, where
  \begin{itemize}
    \item $\frakd$ is an oriented ray in $\R^2$ with endpoint at the origin
    \item $f \in \C[z^\mm]\llbracket t \rrbracket \subset A$ satisfies $f \equiv 1$ (mod $z^\mm t$), where $\mm \in \Z^2$ is the primitive oriented generator of $T\frakd$.
  \end{itemize}
\end{definition}
\NI Here $T\frakd$ denotes the tangent space to $\frakd$, i.e. the line containing $\frakd$. Note that in particular the ray $\frakd$ must point in a rational direction. We will sometimes refer to $f$ as the \hl{label} of the wall.

\begin{example}
 Figure~\ref{fig:walls} illustrates two examples of walls. In both cases the underlying ray points in the diretion $(2,3)$. We will say that the wall is \hl{outgoing} (resp. \hl{incoming}) if the ray is oriented away from (resp. towards) the origin.
\end{example}

\begin{figure}[H]
\centering
\includegraphics[width=0.8\textwidth]{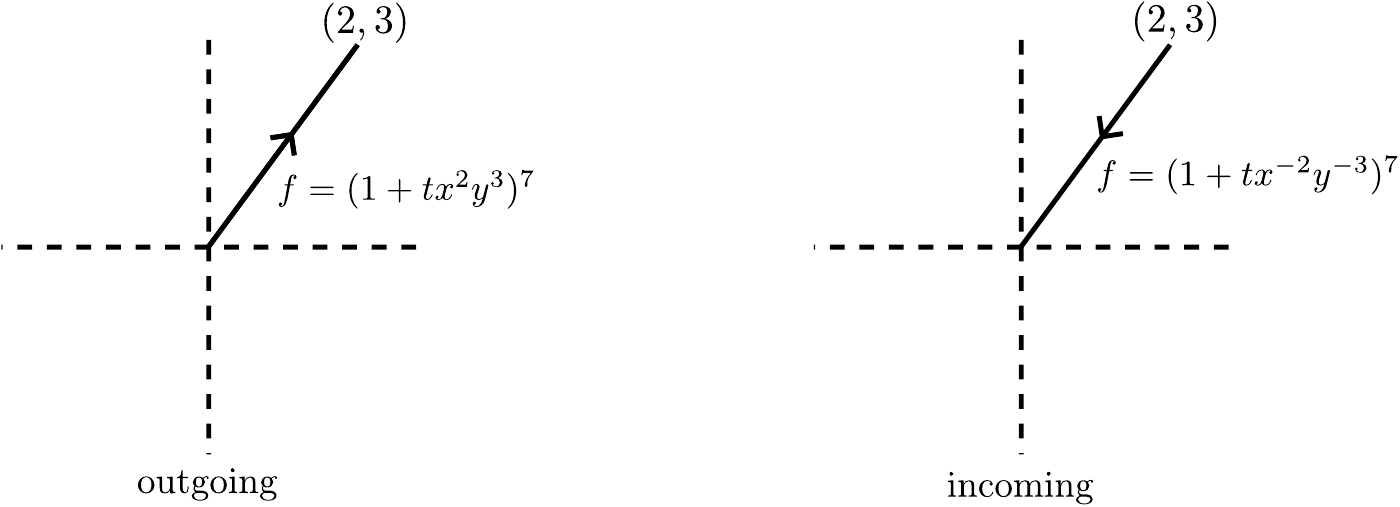}
\caption{Two examples of walls $(\frakd,f)$ in $\R^2$.}
\label{fig:walls}
\end{figure}

\begin{definition}\label{def:scat_diag}
  A \hl{scattering diagram} $\calD$ is a multiset of walls in $\R^2$ such that, for each $k \in \Z_{\geq 1}$, all but finitely many labels appearing in $\calD$ are congruent to $1$ modulo $t^k$.
\end{definition}
\NI In other words, we allow infinitely many walls, but most of the nontrivial terms will involve arbitrarily large powers of $t$.

\begin{remark}
 The definition of scattering diagrams given in Definition~\ref{def:scat_diag} is greatly simplified for our purposes. 
In the literature one can find more general definitions which allow rays not ending at the origin, labels coming from other algebras, walls in $\R^n$ for $n \geq 3$, and so on.
\end{remark}

\begin{definition}
  Given a wall $(\frakd,f)$, its \hl{wall-crossing monodromy} $\Theta_{(\frakd,f)} \in \aut(A)$ is defined on monomials by $z^\mm$, $\mm \in \Z^2$, by
  \[\Theta_{(\frakd,f)}(z^\mm) := f^{\langle \mm,\nn\rangle}z^\mm,\] where $\nn \in \Z^2$ is the oriented primitive normal to $\mm$.
\end{definition}

\begin{example}
For the first example in Figure~\ref{fig:walls}, we have $\nn = (-3,2)$ and $\Theta_{(\frakd,f)}(x^ay^b) = (1 + tx^2y^3)^{7(-3a+2b)}x^ay^b$ for all $(a,b) \in \Z^2$.
\end{example}

\begin{definition}\label{def:tot_mon}
  Given a scattering diagram $\calD$, its \hl{total monodromy} $\mon_\calD \in \aut(A)$ is given by 
  \begin{align*}
   \mon_\calD := \prod_{(\frakd,f) \in \calD} \Theta_{(\frakd,f)},
   \end{align*} 
   where the product is over all walls in $\calD$ in the counterclockwise direction, starting at the positive $x$-axis.
\end{definition}
\NI Note that infinite product in Definition~\ref{def:tot_mon} converges $t$-adically thanks to the finiteness condition in Definition~\ref{def:scat_diag}.

The following proposition is central to the theory of scattering diagrams.
\begin{proposition}[Kontsevich--Soibelman, generalized by Gross--Siebert]\label{prop:KSGS}
Given any scattering diagram $\calD$ containing only incoming walls, one can add outgoing walls (typically infinitely many) to obtain a new scattering diagram $\ks(\calD)$ such that $\mon_{\ks(D)} = \1$.
\end{proposition}
\NI It turns out that the scattering diagram $\ks(D)$ in Proposition~\ref{prop:KSGS} is essentially unique. More precisely,
 it is unique up to the following moves and their inverses: (i) adding a trivial wall with label $f = 1$, or (ii) replacing two walls $(\frakd,f_1),(\frakd,f_2)$ with a single wall $(\frakd,f_1f_2)$ (note that neither of these moves affects the total monodromy).

\subsection{Some examples of scattering diagrams}\label{subsec:toy_examples}

The proof of Proposition~\ref{prop:KSGS} is not difficult: one kills the total monodromy modulo $t^k$ by adding finitely many walls, and does this inductively over all $k$; see e.g. \cite[Thm. 1.4]{gross2010tropical} for more details. Nevertheless, it can be quite challenging and subtle to fully compute $\ks(\calD)$, even for very simple scattering diagrams $\calD$.

\begin{example}
 Let $\calD_{e_1,e_2}^{1,1}$ denote the scattering diagram with two initial incoming walls in directions $(-1,0)$ and $(0,-1)$, with labels $1+tx$ and $1+ty$ respectively.
The Kontsevich--Soibelman algorithm adds $3$ outgoing walls. The first two added walls complete the two initial incoming rays into lines and have the same labels $1+tx$ and $1+ty$, while the third wall points in the $(1,1)$ direction and has label $1+t^2xy$. See Figure~\ref{fig:sd1}.
\end{example}

\begin{figure}[H]
\centering
\includegraphics[width=0.8\textwidth]{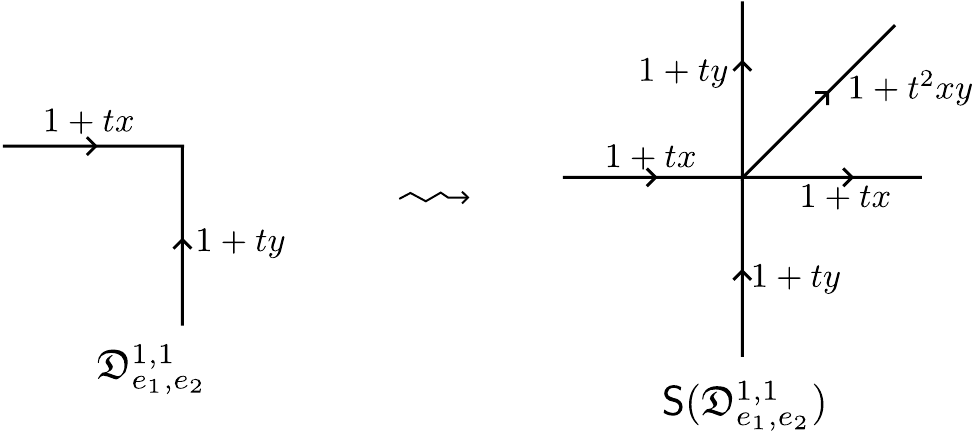}
\caption{The scattering diagram $\calD_{e_1,e_2}^{1,1}$ and the result $\ks(\calD_{e_1,e_2}^{1,1})$ after the Kontsevich--Soibelman algorithm.}
\label{fig:sd1}
\end{figure}

\begin{exercise}
 Check by hand that the total monodromy of the scattering diagram with five rays depicted on the right side of Figure~\ref{fig:sd1} is indeed trivial.
\end{exercise}

\begin{example}
Let $\calD_{e_1,e_2}^{2,2}$ denote the scattering diagram with two inital incoming walls in directions $(-1,0)$ and $(0,-1)$, now with labels $(1+tx)^2$ and $(1+ty)^2$.
In this case, the Kontsevich--Soibelman algorithm adds infinitely many outgoing walls, with directions $(k,k+1)$ and $(k+1,k)$ for all $k \in \Z_{\geq 0}$, as well as an outgoing wall in direction $(1,1)$. In other words, there are two infinite sequences of slopes which converge to $1$ from both sides.
The wall labels are as follows:
\begin{align*}
f_{k,k+1} = (1+t^{2k+1}x^ky^{k+1})^2,\;\;\;\;\; f_{k+1,k} = (1+t^{2k+1}x^{k+1}y^k)^2,\;\;\;\;\; f_{1,1} = (1-t^2xy)^{-4}.
\end{align*}
See Figure~\ref{fig:sd2}.
\end{example}
\begin{figure}[H]
\centering
\includegraphics[width=0.8\textwidth]{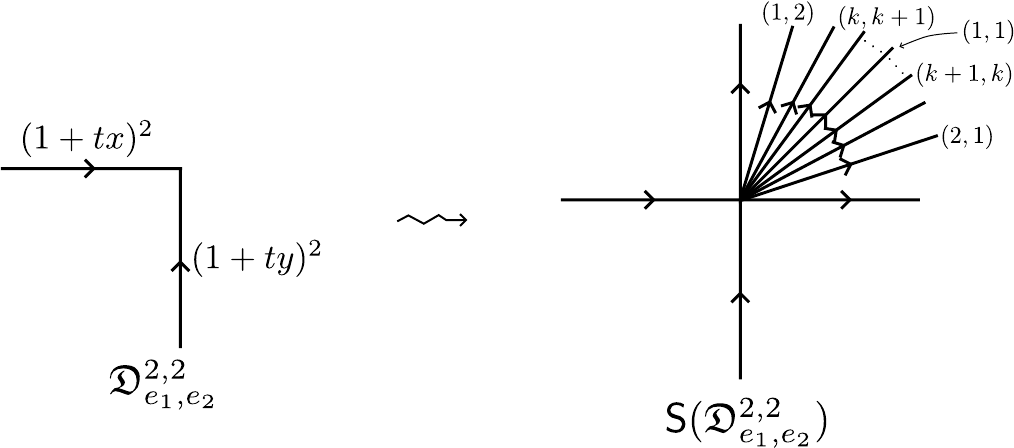}
\caption{The scattering diagram $\calD_{e_1,e_2}^{2,2}$ and the result $\ks(\calD_{e_1,e_2}^{2,2})$ after the Kontsevich--Soibelman algorithm.}
\label{fig:sd2}
\end{figure}

\begin{example}
Our final example is in fact a theorem, first proved by Gross--Pandharipande \cite{gross2010quivers}, and later extended by Gr\"afnitz--Luo \cite{grafnitz2023scattering} following the techniques of Gross--Hacking--Keel--Kontsevich \cite{GHKK2018}.
Starting with $\calD_{e_1,e_2}^{3,3}$, the resulting scattering diagram $\ks(\calD_{e_1,e_2}^{3,3})$ contains an infinite sequence of walls in directions $(1,3),(3,8),(8,21),\dots$ whose slopes converge to $\tfrac{1}{2}(3+\sqrt{5})$ from above, and another infinite sequence of walls in directions $(3,1),(8,3),(21,8),\dots$ whose slopes converge to $\tfrac{1}{2}(3-\sqrt{5})$ from below.
Within the cone spanned by these two limiting rays, there is a wall of every rational slope. See Figure~\ref{fig:sd3}.

\end{example}

\begin{figure}[H]
\centering
\includegraphics[width=0.8\textwidth]{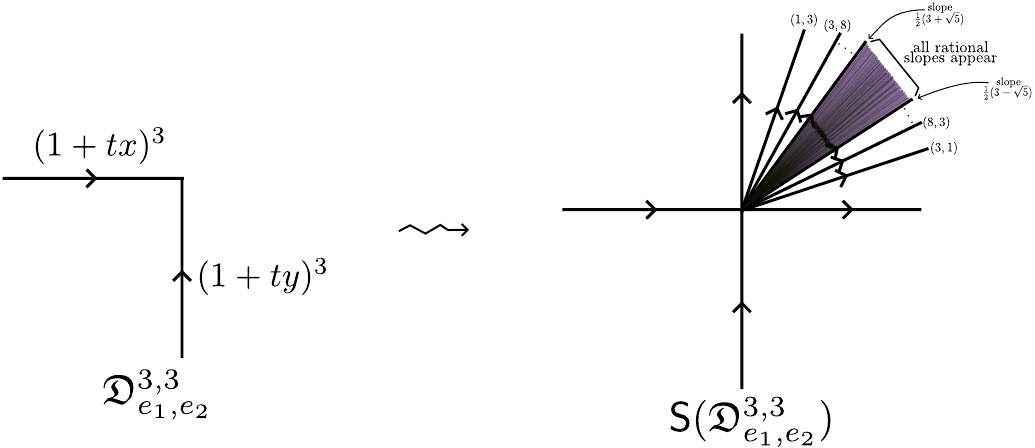}
\caption{The scattering diagram $\calD_{e_1,e_2}^{3,3}$ and the result $\ks(\calD_{e_1,e_2}^{3,3})$ after the Kontsevich--Soibelman algorithm.}
\label{fig:sd3}
\end{figure}

The above examples naturally fit into a more general family as follows.
\begin{notation}\label{not:D_m^l}
Given $\mm_1,\dots,\mm_\ell \in \Z^2$ and $k_1,\dots,k_\ell \in \Z_{\geq 1}$, let $\calD_{\mm_1,\dots,\mm_\ell}^{k_1,\dots,k_\ell}$ denote the scattering diagram with incoming walls $(\R_{\leq 0} \cdot \mm_s,(1+tz^{\mm_s})^{k_s})$ for $s = 1,\dots,\ell$.  
\end{notation}

\subsection{The tropical vertex theorem}\label{subsec:tropical_vertex}

Recall that we seek to construct well-placed curves with respect to a uninodal anticanonical divisor, and Proposition~\ref{proplet:fund_bij} relates these to certain curves in a corresponding edge toric surface.
The tropical vertex theorem states that certain curves in edge toric surfaces are encoded by an associated scattering diagram.
Here is a somewhat informal statement.

\begin{theorembox}[Gross--Pandharipande--Siebert \cite{gross2010tropical}]\label{thm:GPS}
Let $(X_\tor,D_\tor)$ be a toric pair and $(\wt{X}_\tor,\wt{D}_\tor)$ an edge toric blowup with blowup data $\mm_1,\dots,\mm_\ell \in \Z^2$.
Put $\calD := \calD_{\mm_1,\dots,\mm_\ell}^{1,\dots,1}$. Then the following coincide:
\begin{itemize}
  \item the log coefficients of wall labels in $\ks(\calD)$
  \item the ``counts'' of rational curves in $\wt{X}_\tor$ which intersect $\wt{D}_\tor$ in one point.
\end{itemize}
\end{theorembox}

More precisely, for primitive $(a,b) \in \Z^2$, if the ray $\R_{\geq 0} \cdot (a,b)$ has label $f_{a,b} \in \C[z^{(a,b)}]\llbracket t \rrbracket$, then the coefficient of $z^{\ka(a,b)}$ in $\log f_{a,b}|_{t=1}$ is a relative (or log) Gromov--Witten count of rational curves in $\wt{X}_\tor$ which intersect $\wt{D}_\tor$ in a single smooth point of the divisor component associated to the fan ray $\R_{\geq 0} \cdot W_\tor(a,b)$, with intersection multiplicity $\ka$. Note that we can always assume that there is such a divisor component of $\wt{D}_\tor$, at least after some additional corner blowups.

\begin{remark}
Since we do not require $\mm_1,\dots,\mm_\ell$ to be pairwise distinct, $\calD_{\mm_1,\dots,\mm_\ell}^{1,\dots,1}$ can sometimes be written more succinctly, e.g. $\calD_{(1,0),(1,0),(0,1)}^{1,1,1} = \calD_{(1,0),(0,1)}^{2,1}$ (c.f. Figure~\ref{fig:fund_bij}).
\end{remark}

Typically the counts in the second bullet of Theorem~\ref{thm:GPS} will be quite virtual, with nontrivial contributions from multiple covers.
However, we are most interested in those curves in $\wt{X}_\tor$ which relate to $(p,q)$-well-placed curves for $\gcd(p,q) = 1$, and in this case one can show that the virtual counts actually agree with the naive ones, with every curve counting positively.
Thus, combining Theorem~\ref{thm:GPS} with Proposition~\ref{proplet:fund_bij} (along with some technical lemmas which we suppress here) essentially gives the following. 

\begin{corollary}
  Let $(X,D)$ be a uninodal Looijenga pair, and let $\calT$ be a toric model with blowup data $\mm_1,\dots,\mm_\ell \in \Z^2$.
  Put $\calD_\calT := \calD_{\mm_1,\dots,\mm_{\ell}}^{1,\dots,1}$. The following are equivalent:
  \begin{itemize}
    \item there exists a $(p,q)$-well-placed rational algebraic curve in $(X,D)$
    \item the coefficient of $z^{W_\calT(p,q)}$ in $\ks(\calD_\calT)$ is nonzero.
  \end{itemize}
\end{corollary}
Recall here that $W_\calT$ is the bijection from $\Z^2_{\geq 0}/\sim$ and $\Z^2$ as in Proposition~\ref{proplet:fund_bij}.
The rough upshot is that, in order to complete the proof of Theorem~\ref{thmlet:C}, we just need to determine the walls in the scattering diagram $\ks(\calD_{(-1,-3),(1,0)}^{1,1})$ (or more precisely those whose label has nonvanishing lowest order term).

\sss

We end this lecture with a very brief informal sketch of the proof of Theorem~\ref{subsec:tropical_vertex}, using the language of neck-stretching from symplectic field theory (note that the actual proof in \cite{gross2010tropical} uses the degeneration formula in relative Gromov--Witten theory).
Let $U$ be a small neighborhood of the toric divisor $D_\tor$ in $X_\tor$ with contact type boundary, and let $\wt{U}$ be the corresponding neighborhood in $\wt{X}_\tor$ which contains $\wt{D}_\tor$ and also all of the exceptional divisors of the edge blowups.
The main idea is to neck-stretch $\wt{X}_\tor$ along $\bdy \wt{U}$ in order to decompose the curve counts in the second bullet of Theorem~\ref{subsec:tropical_vertex} into simpler pieces.
The resulting pieces are as follows:
\begin{enumerate}[label=(\arabic*)]
  \item curves in $X_\tor$ which intersect $D_\tor$ in multiple points with prescribed intersection multiplicities
  \item multiple covers of trivial cylinders in the ``normal bundle'' $N(\wt{D}_\tor)$ of $\wt{D}_\tor \subset \wt{X}_\tor$.
\end{enumerate}
It is well-known that the contributions from (1) are given by counts of tropical curves in $\R^2$ with prescribed asymptotics. 
Meanwhile, the (virtual) contributions from (2) are computed explicitly in \cite[Prop 5.2]{gross2010tropical}, following \cite{bryanpandharipande_CY3} (these are typically rational numbers, i.e. not necessarily positive or integral).
Gluing together (1) and (2) gives the counts in the second bullet of Theorem~\ref{subsec:tropical_vertex} as an explicit signed weighted count of tropical curves. 
One then relates this to the counts in the first bullet of Theorem~\ref{subsec:tropical_vertex} using the ``deformations of scattering diagrams'' technique discussed in \cite[\S1.4]{gross2010tropical}.

\begin{remark}
 For the uninodal Looijenga pair $(\CP^2,N_0)$ (recall Example~\ref{ex:CP2_N0}) with its toric model $\calT_0$ (Example~\ref{ex:T0}) and associated scattering diagram $\calD_{\calT_0} = \calD_{(-1,3),(1,0)}^{1,1}$, the symmetries $\Phi,\Psi$ for well-placed curves translate into symmetries of the scattering diagram $\ks(\calD_{\calT_0})$. 
\end{remark}

\section{Putting it all together}\label{sec:putting_together}

In this lecture, we first introduce in \S\ref{prop:change_of_lattice} one final ingredient needed to prove Theorem~\ref{thmlet:C}, the so-called ``change of lattice trick''. 
We then readily complete the proof of Theorem~\ref{thmlet:C} in \S\ref{subsec:proof_of_Thm_C}.
We end in \S\ref{subsec:del_Pezzo} by formulating a generalization of Theorem~\ref{thmlet:A} to del Pezzo surfaces.

\subsection{The change of lattice trick}\label{subsec:change_of_lattice}

It is natural to generalize the definition of scattering diagram given in Definition~\ref{def:scat_diag} by replacing $\R^2$ with the real vector space $\MM_\R := \MM \otimes \R$, where $\MM$ is a rank two lattice (i.e. an abelian group which is abstractly isomorphic to $\Z^2$). In this case the wall-crossing monodromies depend implicitly on the lattice $\MM$. In particular, passing to a finite rank sublattice $\MM' \subset \MM$ really gives a different scattering diagram, even though this is a natural isomorphism $\MM'_\R \cong \MM_\R$.

In the case of scattering diagram of the form $\calD_{\mm_1,\mm_2}^{\ell_1,\ell_2}$ (recall Notation~\ref{not:D_m^l}), by replacing $\Z^2$ with the sublattice $\lan \mm_1,\mm_2\ran$ spanned by $\mm_1,\mm_2$ and pulling back by an isomorphism $\Z^2 \ra \lan \mm_1,\mm_2\ran$, we can convert the rays $\R_{\leq 0} \cdot \mm_1$ and $\R_{\leq 0} \cdot \mm_2$ to $\R_{\leq 0} \cdot e_1$ and $\R_{\leq 0} \cdot e_2$ as in the examples in \S\ref{subsec:toy_examples}, at the cost of changing the exponents appearing in the labels. This is called the ``change of lattice trick'' in \cite[\S C.3]{GHKK2018} (see also \cite[\S1.2]{grafnitz2023scattering} and \cite[\S6.1]{sesqui}).
Concretely, the following reduces the computation of $\ks(\calD_{\mm_1,\mm_2}^{\ell_1,\ell_2})$ to that of $\ks(\calD_{e_1,e_2}^{d\ell_1,d\ell_2})$ for $d := \mm_1 \wedge \mm_2$.
\begin{propositionbox}\label{prop:change_of_lattice}
  Fix positive integers $\ell_1,\ell_2 \in \Z_{\geq 1}$ and primitive noncolinear vectors $\mm_1,\mm_2 \in \Z^2$ with determinant
  $d := \mm_1 \wedge \mm_2 \in \Z$. The rays in $\ks(\calD_{\mm_1,\mm_2}^{\ell_1,\ell_2})$ are precisely the images of the rays in $\ks(\calD_{e_1,e_2}^{d\ell_1,d\ell_2})$ under the map $\R^2 \ra \R^2$ sending $e_1$ to $\mm_1$ and $e_2$ to $\mm_2$.
\end{propositionbox}
Moreover, for any $\mm = a\mm_1 + b \mm_2$ with $a,b \in \Z$, the coefficient of $z^\mm$ in $\ks(\calD_{\mm_1,\mm_2}^{\ell_1,\ell_2})$ agrees (up to a small combinatorial factor) with the coefficient of $z^{(a,b)}$ in $\ks(\calD_{e_1,e_2}^{d\ell_1,d\ell_2})$, while the coefficient of $z^\mm$ vanishes if $\mm \notin \lan \mm_1,\mm_2\ran$.

\begin{example}
For our main example $(\CP^2,N_0)$ with toric model $\calT_0$ from Example~\ref{ex:T0}, we have $\mm_1 = (-1,-3)$, $\mm_2 = (1,0)$, and $d := \mm_1 \wedge \mm_2 = 3$.
Thus the rays in $\ks(\calD_{\calT_0})$ are the images of the rays in $\ks(\calD_{e_1,e_2}^{3,3})$ under the linear transformation with matrix
$A =\begin{pmatrix}
-1 & 1\\ -3 & 0 
\end{pmatrix}$.
\end{example}

\begin{upshot}\label{upshot:rays_from_3_3}
For each ray $\R_{\geq 0} \cdot (a,b)$ in $\ks(\calD_{e_1,e_2}^{3,3})$, we get a ray $\R_{\geq 0} \cdot (a\mm_1 + b \mm_2)$ in $\ks(\calD_{\calT_0})$, and thus a $W_{\calT_0}^{-1}(a\mm_1 + b\mm_2)$-well-placed curve in $(\CP^2,N_0)$.
\end{upshot}

\subsection{Completing the proof of Theorem~\ref{thmlet:C}}\label{subsec:proof_of_Thm_C}

The following exercise now completes the proof of the existence part of Theorem~\ref{thmlet:C}.

\begin{exercise}
Under Upshot~\ref{upshot:rays_from_3_3} and the description of $W_{\calT_0}$ from Example~\ref{ex:W_for_T0}, the set of coprime $p \geq q \geq 1$ for which we get a $(p,q)$-well-placed curve corresponds precisely to cases (a) and (b) in Theorem~\ref{thmlet:C}, i.e. (a) $p/q = \tfrac{\fib_{2k+5}}{\fib_{2k+1}}$ and (b) $p/q > \tau^4$.
\end{exercise}

Meanwhile, the uniqueness part of Theorem~\ref{thmlet:C} is elementary and requires only the adjunction formula.

\begin{exercise}
 Show that any index zero rational $(p,q)$-sesquicuspidal curve in $\CP^2$ of degree $d$ must satisfy the following conditions:
 \begin{itemize}
   \item {\em index zero}: $d = \tfrac{1}{3}(p+q)$
   \item {\em adjunction}: $(d-1)(d-2) \geq (p-1)(q-1)$,
 \end{itemize}
 and that these imply the Diophantine inequality
 \begin{align}\label{eq:dioph_ineq}
 p^2 + q^2 - 7pq + 9 \geq 0.
 \end{align}
Show that the solutions to \eqref{eq:dioph_ineq} with coprime $p \geq q \geq 1$ precisely correspond to cases (a) and (b) in Theorem~\ref{thmlet:C}.
\end{exercise}

\subsection{Monotone del Pezzo surfaces}\label{subsec:del_Pezzo}
 
Although we have mostly focused on constructing sesquicuspidal curves in the complex projective plane, much of the formalism in these notes extends to any rational surface with a uninodal anticanonical divisor, working in either the complex or symplectic category. 
One natural extension is to consider \hl{del Pezzo surfaces},
which by definition are those nonsingular complex projective surfaces for which the anticanonical bundle is ample.
Up to diffeomorphism, these are $\CP^2 \# (\underbrace{\ovl{\CP}^2 \# \cdots \# \ovl{\CP}^2}_j)$ for $j = 0,1,\dots,8$ (i.e. $j$-point blowup of $\CP^2$) and $\CP^1 \times \CP^1$.
It is a fact that, up to symplectomorphism, each del Pezzo surface $M$ admits a unique symplectic form $\om_X$ which is \hl{unimonotone}, i.e. $[c_1(X)] = [\om_X] \in H^2(M;\R)$, and moreover any unimonotone symplectic four-manifold is symplectomorphic to one of these (see e.g. \cite{salamon2013uniqueness}).

Recall that the \hl{degree} of a del Pezzo surface $M$ is by definition the self-intersection number of the anticanonical class, i.e. $c_1(M)^2 \in \{9,\dots,0\}$.
For the del Pezzo surfaces with degree $c_1(M)^2 \in \{9,8,7,6,5\}$, it turns out that the biholomorphism classification coincides with the diffeomorphism classification, and we refer to these del Pezzo surfaces as \hl{rigid}.
For a rigid del Pezzo surface $M$, equipped with its unimonotone symplectic form, it is shown in \cite{cristofaro2020infinite} that $\emb_M(1,a)$ contains an explicit infinite staircase closely analogous to the Fibonacci staircase $\emb_{\CP^2}(1,a)$ for $a \in [1,\tau^4]$, and we denote the accumulation point by $a_\acc^M \in (1,\infty)$ (e.g. $a_\acc^{\CP^2} = \tau^4$).

\begin{theoremconjecturebox}\label{thmconj:del_Pezzo}
For any unimonotone symplectic four-manifold $M^4$ and $N \in \Z_{\geq 1}$, we have:
\begin{itemize}
  \item for $c_1(M)^2 \in \{9,8,7,6,5\}$, $\emb_M^N(1,a) = 
\begin{cases}
  \emb_M(1,a) & 1 \leq a \leq a_\acc^M \\
  \frac{a}{a+1} & a > a_\acc^M
\end{cases}
  $
\item for $c_1(M)^2 \in \{4,3,2,1\}$, $\emb_M^N(1,a) = \frac{a}{a+1}$ for all $a \geq 1$.
\end{itemize}
\end{theoremconjecturebox}

\NI More precisely, the lower bounds in all cases are proved in \cite{sesqui} using the above scattering diagram techniques, and the symplectic embeddings needed for the corresponding upper bounds are known in the cases $c_1(M)^2 \in \{9,8,7,6\}$ (via \cite[Prop. 3.1]{cristofaro2022higher}) but not in the remaining cases. 

\begin{remark}
  The proof of the obstructive part of Theorem / Conjecture ~\ref{thmconj:del_Pezzo} proceeds via an analogue of Theorem~\ref{thmlet:C} for sesquicuspidal curves in rigid del Pezzo surfaces.
However, an important subtlety is that we only classify those $(p,q)$ for which there exists a rational index zero $(p,q)$-sesquicuspidal curve with $\gcd(p,q) = 1$, and {\em not} the homology classes of such curves.
A refined classification result which also records homology classes appears to be important for the general (unrestricted) four-dimensional stabilized ellipsoid embedding problem, and is closely related to understanding refinements of the scattering diagrams $\ks(\calD_{\mm_1,\dots,\mm_\ell}^{k_1,\dots,k_\ell})$ in which we work over $\C[x,x^{-1},y,y^{-1}]\llbracket t_1,\dots,t_k\rrbracket$ rather than $\C[x,x^{-1},y,y^{-1}]\llbracket t\rrbracket$.
\end{remark}




\bibliographystyle{alpha}
\bibliography{biblio}

\end{document}